\documentclass[A4paper,11pt,english]{article}
\usepackage{babel}
\usepackage{graphicx}
\usepackage{amsmath}
\usepackage[utf8]{inputenc}
\usepackage{amsfonts}
\usepackage{amssymb}
\usepackage{amsmath}
\usepackage{amsthm}
\usepackage{array}
\usepackage{hyperref}
\title{\bf On the convergence properties of Durrmeyer-Sampling Type Operators in Orlicz spaces}
\author{ {\bf Danilo Costarelli} \hskip1cm {\bf Michele Piconi} \hskip1cm {\bf Gianluca Vinti} \\ 
Department of Mathematics and Computer Science \\
University of Perugia\\
1, Via Vanvitelli, 06123 Perugia, Italy \\ 
{\small {\tt danilo.costarelli@unipg.it}} - {\small {\tt michele.piconi1211@gmail.com}} - {\small {\tt gianluca.vinti@unipg.it}} }
\date{}

\newtheorem{lemma}{Lemma}[section]
\newtheorem{teorema}{Theorem}[section]
\newtheorem{cor}{Corollary}[section]
\theoremstyle{definition}
\newtheorem{remark}{Remark}[section]

\begin{document}

\maketitle

\begin{abstract}
 Here we provide a unifying treatment of the convergence of a general form of sampling type operators, given by the so-called Durrmeyer sampling type series.
In particular we provide a pointwise and uniform convergence theorem on $\mathbb{R}$, and in this context we also furnish 
 a quantitative estimate for the order of approximation, using the modulus of continuity of the function to be approximated.
 Then we obtain a modular convergence theorem in the general setting of Orlicz spaces $L^\varphi(\mathbb{R})$. From the latter result, the convergence in $L^p(\mathbb{R})$-space, $L^\alpha\log^\beta L$, and the exponential spaces follow as particular cases. Finally, applications and examples with graphical representations are given for several sampling series with special kernels. 
\end{abstract}
\medskip\noindent
{\small {\bf AMS subject classification:}
41A25,41A35,46E30,47A58,47B38,94A12} \newline
{\small {\bf Key Words:}  Generalized sampling operators, Orlicz spaces,
modular convergence, order of approximation, Durrmeyer sampling operators}

\section{Introduction}
The theory of sampling series, in one and several variables, is one of the most studied topics in the approximation theory, in view of its many applications, especially in signal and image processing.
\\Sampling-type operators have been introduced in order to study approximate version of the well-known Wittaker-Kotel'nikov-Shannon sampling theorem (see, e.g., \cite{20, 1993aa, 2008dl, KKS1, KOPR1, KS1}). Among the most studied families of sampling operators, we can find the celebrated family of the generalized (see, e.g., \cite{19,16,17,18,2001v,2011ade,2020bjma}) and Kantorovich type series (see, e.g., \cite{1, ORTA1}), that have been introduced in the 80s and in 2007, respectively, thanks to the crucial contribution of the German mathematician P.L. Butzer and his coauthors.
\\
The aim of this work is to extend the main approximation properties, including convergence results and quantitative estimates, for the so-called Durrmeyer sampling type series, introduced by C. Bardaro and I. Mantellini in \cite{12}. Durrmeyer sampling operators represent a generalization of the generalized and of the Kantorovich sampling series. The present study is not confined only to the setting of continuous (or uniformly continuous) functions, but it is also extended to the case of functions belonging to Orlicz spaces. Such spaces have been introduced in the 30s thanks to the Polish mathematician W. Orlicz, as a natural extension of $L ^ p$ spaces, and other useful spaces very used in Functional Analysis and its applications, such as interpolation and exponential spaces. Thus,  in this paper we provide a unifying theory, not only in the sense that Durrmeyer sampling type series represent a generalization of the above sampling type operators, but also since the main convergence results will be given in the general setting of Orlicz spaces.
\\From the literature, it is well-known that the classical Bernstein polynomials

\[\left(B_nf\right)(x):=\sum_{k=0}^{n}p_{n,k}(x)f\left(\frac{k}{n}\right),\,\,\,p_{n,k}=\binom{n}{k}x^k(1-x)^{n-k},\,\,\,x\in[0, 1],\]
\noindent
are known since 1912 and they have been used in order to give one of the most elegant proof of the Weierstrass approximation theorem by algebraic polynomials in the space of the continuous functions over the interval $[0, 1]$. The Durrmeyer method applied to Bernstein polynomials is obtained replacing the sample value $f\left(\frac{k}{n}\right)$ by an integral in which the same generating kernel $p_{n,k}$ appears, i.e.,

\[\left(D_nf\right)(x):=(n+1)\sum_{k=0}^{n}p_{n,k}(x)\int_{0}^{1}p_{n,k}(u)f(u)du,\,\,\,x\in[0, 1].\]
The literature about this operator and its generalizations is very wide; we quote here e.g. \cite{13,14,15,2012rm}.\\
In this paper we apply the Durrmeyer method to the sampling series in a generalized form, considering operators of the following type,

\[ \left ( S_w^{\varphi,\psi}f\right)(x):= \sum_{k\in\mathbb{Z}} \varphi(wx-k) w \int_{\mathbb{R}}\psi(wu-k)f(u)du,\,\,\,x\in\mathbb{R},\]
for $w>0$, in which we replace the integral means by a general convolution integral. In fact, in the Durrmeyer sampling type series, the kernel functions $\varphi$ and $\psi$ satisfy certain moments conditions, together with suitable singularity assumptions, in both continuous and discrete form.\\
The central goal of this paper is to provide theoretical results on Durrmeyer sampling type series, starting from pointwise and uniform convergence of $S_w^{\varphi,\psi}f$ to $f$ on $\mathbb{R}$, assuming $f$ continuous or uniformly continuous and bounded. In this regard, we also investigate the problem of the order of the uniform convergence, in case of uniformly continuous and bounded functions, and in the latter setting, we estimate the order of approximation by means of a quantitative estimate, using the first order modulus of continuity of $f$. The qualitative rate of convergence is also deduced assuming $f$ in suitable Lipschitz classes. Further, we also study the problem of the convergence of $S_w^{\varphi,\psi}f$ to $f$ in the general setting of Orlicz spaces.
Here, we consider the most natural notion of convergence, that is the so-called ``modular convergence'', introduced by the modular functional defined on the space. One of the main advantage in studying approximation theorems in Orlicz spaces is the possibility to approximate not-necessarily continuous signals. This is what usually occurs in real world applications (see, e.g., \cite{2018ae}), in which signals are not very regular (as happens, e.g., for images). In the context of Orlicz spaces, we firstly prove a modular inequality for the operators $S_w^{\varphi,\psi}$, and so, we establish a modular convergence theorem.
At the end of the paper, several examples of kernels $\varphi$ and $\psi$ have been provided together with numerical examples and graphical representations.


\section{Preliminaries and notations}
We denote by $C(\mathbb{R})$~the space of all uniformly continuous and bounded functions 
$f: \mathbb{R} \rightarrow \mathbb{R},$~by $C_c(\mathbb{R})$~the subspace of $C(\mathbb{R})$~whose 
elements have compact support. Moreover by $M(\mathbb{R})$~we denote the space of all 
(Lebesgue) measurable real functions over $\mathbb{R}$.

Let $\varphi : \mathbb{R}^+_0 \rightarrow \mathbb{R}^+_0$~be a convex $\varphi$-function, i.e.,
$\varphi$~satisfies the following assumptions:
\begin{enumerate}
\item $\varphi$~is convex in $\mathbb{R}^+_0$;
\item $\varphi (0) = 0$ and $\varphi (u) > 0,$~for every $u >0$. 
\end{enumerate}
Let us consider the functional
$$I^{\varphi}[f] := \int_{I\!\!R}\varphi (|f(x)|)dx, ~~f \in M(\mathbb{R}).$$
As it is well-known (see, e.g., \cite{8,10,11}), $I^{\varphi}$~is a convex modular 
functional on $M(\mathbb{R})$~and the Orlicz space generated by $\varphi$~is defined by
$$L^{\varphi}(\mathbb{R}) = \{f \in M(\mathbb{R}): I^{\varphi}[\lambda f] < + \infty,
\mbox{for some}~\lambda >0\}.$$
The Orlicz space $L^{\varphi}(\mathbb{R})$~is a vector space and the vector subspace
$$E^{\varphi}(\mathbb{R}) = \{f \in M(\mathbb{R}): I^{\varphi}[\lambda f] < + \infty,
\mbox{for every}~\lambda >0\},$$
is called the space of all finite elements of $L^{\varphi}(\mathbb{R}).$ In general 
$E^{\varphi}(\mathbb{R})$~is a proper subspace of $L^{\varphi}(\mathbb{R})$~and they coincide 
if $\varphi$~satisfies the well-known $\Delta_2$-condition, i.e., if there exists a constant $M>0$ such that
\[\varphi(2u)\le M\varphi(u)\,\,\,\,(u\in\mathbb{R}^+_0).\]
Examples of functions $\varphi$ satisfying the $\Delta_2$-condition are $\varphi(u)=u^p$, $1\le p <+\infty$, or $\varphi_{\alpha,\beta}(u)=u^\alpha\log^\beta(e+u)$, for $\alpha\ge 1$ and $\beta>0$, which generate respectively, the $L^p$-spaces and the Zygmund spaces $L^\alpha\log^\beta L$.
On the other hand, the $\varphi$-function $\varphi_\alpha(t)=e^{t^\alpha}-1$, $\alpha>0$, generates the so-called exponential spaces, which are examples of Orlicz spaces for which $E^{\varphi_\alpha}(\mathbb{R})\subset L^{\varphi_\alpha}(\mathbb{R})$.

In $L^{\varphi}(\mathbb{R})$~we work with a notion of convergence called 
\textit{modular convergence}: we will 
say that a net of functions $(f_w)_{w >0} \subset L^{\varphi}(\mathbb{R})$~
is \textit{modularly convergent} to a function $f \in L^{\varphi}(\mathbb{R})$~if 
$$\lim_{w \rightarrow +\infty}I^{\varphi}[\lambda (f_w -f)] = 0,$$
for \textit{some} $\lambda >0.$~This notion induces a topology in $L^{\varphi}(\mathbb{R}),$~
called \textit{modular topology}. 

In the space $L^{\varphi}(\mathbb{R})$~we can also introduce a norm (the Luxemburg norm), 
defined by
$$\|f\|_{\varphi} := \inf\{\lambda >0 : I^{\varphi}[f/\lambda] \leq 1\}.$$
Thus, we also have a stronger notion of convergence in $L^{\varphi}(\mathbb{R}),$~namely the 
norm convergence. It is well known that $\|f_w - f\|_{\varphi} \rightarrow 0$, as $w\rightarrow +\infty$, if and 
only if $I_{\varphi}[\lambda (f_w-f)] \rightarrow 0$, as $w\rightarrow +\infty$, for \textit{every} $\lambda >0.$~The two 
notions of convergence are equivalent if and only if the function $\varphi$~satisfies 
the $\Delta_2$-condition. 
For further details in the matter, see, e.g., \cite{6,21,9,10,11}.

\section{The generalized Durrmeyer sampling series}
Here we recall the definition of the family of the generalized Durrmeyer sampling operators. Such operators have been firstly introduced in \cite{12} in order to study asymptotic expansion and Voronovskaja-type theorems in case of sufficiently regular functions.
Let us consider two functions $\varphi, \psi\in L^1(\mathbb{R})$, such that $\varphi$ is bounded in a neighborhood of the origin, and satisfying
\begin{equation}\label{IpotesiD}
\sum_{k\in\mathbb{Z}} \varphi(u-k)=1,\text{ for every $u\in\mathbb{R}$, and}  \int_{\mathbb{R}} \psi(u)du=1.
\end{equation}
Note that, $\psi$ defines an approximate identity (see, e.g., \cite{1971ap, 2008itsf, 2016jfa, 2017pam}) by the formula $\psi_w(u):=w\psi(wu)$, $u\in\mathbb{R}$ and $w>0$.\\
For any $\nu \in \mathbb{N}_0$, let us define the discrete and continuous algebraic moments of $\varphi$ and $\psi$ respectively, as follows
\[ m_{\nu}(\varphi,u):=\sum_{k\in\mathbb{Z}} \varphi(u-k)(k-u)^\nu,\quad u\in\mathbb{R},\] \\ and  \[\tilde {m}_\nu\left(\psi\right):= \int_{\mathbb{R}}t^\nu \psi(t)dt\] \\and the discrete and continuous absolute moments as \\ 
\begin{equation}\label{momentoassolutod} M_{\nu}(\varphi):=\sup_{u\in\mathbb{R} }\sum_{k\in\mathbb{Z}}\left | \varphi(u-k)\right|\left|u-k\right |^\nu \end{equation} \\and\\  \[\tilde {M}_\nu(\psi):= \int_{\mathbb{R}}\left |t\right|^\nu\left | \psi(t)\right|dt,\]
respectively.\\
Note that, for a function $f:\mathbb{R}\to\mathbb{R}$, the definition of the moments $M_\nu(f), \tilde{M}_\nu(f)$ can also be given for any $\nu\ge 0$.\\
Now, we will call \textit{kernels} a pair of functions $\varphi$ and $\psi$ belonging to $L^1(\mathbb{R})$, satisfying (\ref{IpotesiD}), and such that, there exists $r>0$ for which $M_r(\varphi)<+\infty$.\\
For $w>0$ and for kernels $\varphi$ and $\psi$, we define a family of operators $\left ( S_w^{\varphi,\psi}\right)_{w>0}$  by
 \begin{equation}\label{Durrmeyeroperator} \left ( S_w^{\varphi,\psi}f\right)(x)= \sum_{k\in\mathbb{Z}} \varphi(wx-k) w\int_{\mathbb{R}}\psi(wu-k)f(u)du,\text{$\qquad$} x\in \mathbb{R},
 \end{equation}
 for any given function $f$ such that the above series is convergent, for every $x\in\mathbb{R}$.
 $S_w^{\varphi,\psi}$ are called the \textit{Durrmeyer sampling operators} based on $\varphi$ and $\psi$.\\
In order to study convergence results for the Durrmeyer sampling operators, we first show the following lemma.
 \begin{lemma}\label{Lemma3} Under the above assumptions on the kernel $\varphi$, we have
\begin{description} 
\item[(i)]$M_{0}(\varphi):=\sup_{u\in\mathbb{R} }\sum_{k\in\mathbb{Z}}\left | \varphi(u-k)\right | <+\infty;$
\item[(ii)]For every $\gamma>0$, \[\lim_{w\rightarrow +\infty}\sum_{\left |wx-k\right|>\gamma w}^{}\left | \varphi(wx-k)\right|=0,\]uniformly with respect to $x\in\mathbb{R}$.
\end{description} 
 \end{lemma}
For a proof of Lemma \ref{Lemma3}, see, e.g., \cite{1}.
\begin{remark}\label{osservazione3.1}We note that:
\begin{description}
\item[(a)] for $\mu,\nu>0$ with $\mu\le\nu$, then $M_\nu(\varphi)<+\infty$ implies $M_\mu(\varphi)<+\infty$, see, e.g., \cite{2019jat}. Moreover, if $\varphi$ has compact support, we immediately have that $M_\nu(\varphi)<+\infty$, for every $\nu\ge 0$.\\
Finally, in an analogous way, for $\mu, \nu>0$ with $\mu\le\nu$, $\tilde{M}_\nu(\psi)<+\infty$ implies $\tilde{M}_\mu(\psi)<+\infty$. 
\item[(b)] from Lemma \ref{Lemma3}, $ S_w^{\varphi,\psi}f$ are well-defined for every $f \in L^\infty(\mathbb{R})$. Indeed,  \[  \left|\left ( S_w^{\varphi,\psi}f\right)(x)\right | \le M_{0}(\varphi)\tilde{M_{0}}(\psi)\lVert f \rVert_\infty,\ x \in \mathbb{R}.\]Thus, the Durrmeyer sampling operator is a bounded linear operator mapping $L^\infty(\mathbb{R})$ into itself.
\item[(c)] if a function $\psi$ is bounded in a neighborhood of the origin and $\psi(u)=\mathcal{O}(\left|u\right|^{-\alpha})$, as $\left|u\right|\rightarrow+\infty$, with $\alpha>\nu+1$, $\nu>0$, then:\[M_\mu(\psi)<+\infty,\text{ for every }0\le\mu\le\nu,\]
see, e.g., \cite{2019jat}.
\end{description}
\end{remark}

\section{Convergence theorems}
From now on, in the whole paper we will always consider kernels $\varphi$ and $\psi$ satisfying the assumptions introduced in Section 3. Note that, with the name \textit{kernels}, we refer to both the functions $\varphi$ and $\psi$, even if they satisfy different assumptions.\\Now, we prove the following pointwise and uniform convergence theorem.
 \begin{teorema}\label{convergence}
  Let $f\in L^\infty(\mathbb{R})$. Then  \[ \lim_{w\rightarrow +\infty}( S_w^{\varphi,\psi}f)(x)=f(x)\] \\at any point $x$ of continuity of $f$. Moreover, if $f\in C(\mathbb{R})$, then \[ \lim_{w\rightarrow +\infty}\lVert S_w^{\varphi,\psi}f-f\rVert_\infty=0.\]
  \end{teorema} 
  \begin{proof}
We only prove the second part of the theorem, since the first part can be obtained by similar methods. Let $\varepsilon>0$ be fixed. Then there exists $\delta>0$ such that $\left|f(x)-f(y)\right|<\varepsilon$ when $\left|x-y\right|<\delta$. Let $x \in \mathbb{R}$ be fixed. Using (\ref{IpotesiD}), we have
\[\left |( S_w^{\varphi,\psi}f)(x)-f(x)\right |= \left | \sum_{k\in\mathbb{Z}} \varphi(wx-k) w \int_{\mathbb{R}}\psi(wu-k)f(u)du-f(x) \right |\]
\\
\[= \left | \sum_{k\in\mathbb{Z}} \varphi(wx-k) w \int_{\mathbb{R}}\psi(wu-k) \left [ f(u)-f(x) \right ]du \right |\]
\\
\[\le  \sum_{k\in\mathbb{Z}}\left | \varphi(wx-k) \right | w \int_{\mathbb{R}} \left |\psi(wu-k)\right | \left | f(u)-f(x) \right | du \]
\\
\[= \left \{  \sum_{\left | wx-k \right| \le \frac{\delta}{2}w}\ +\   \sum_{\left | wx-k \right| > \frac{\delta}{2}w}  \right \} \left | \varphi(wx-k) \right | w \int_{\mathbb{R}} \left |\psi(wu-k)\right | \left | f(u)-f(x) \right | du \]
\\
$=I_1 +  I_2.$\\
The first term can be further divided into
\[I_1 = \sum_{\left | wx-k \right| \le \frac{\delta}{2}w}  \left | \varphi(wx-k) \right | w\left \{ \int_{\left | wu-k \right| < \frac{\delta}{2}w}\ + \ \int_{\left | wu-k \right| \ge \frac{\delta}{2}w} \right \} \left |\psi(wu-k)\right | \left | f(u)-f(x) \right | du \]
\\
$=I_{1,1}+I_{1,2}.$
\\
For $u \in\mathbb{R}$ such that $\left |wu-k \right| < \frac{\delta}{2}w$, if $\left | wx-k \right | \le \frac{\delta}{2}w$, we have \[ \left | u-x \right| \le \left |u- \frac{k}{w} \right |+\left |  \frac{k}{w}-x \right | < \frac{\delta}{2}+ \frac{\delta}{2}=\delta. \]
Thus, \[ I_{1,1}<\varepsilon\sum_{\left | wx-k \right| \le \frac{\delta}{2}w}  \left | \varphi(wx-k) \right | w \int_{\left | wu-k \right| < \frac{\delta}{2}w}\  \left |\psi(wu-k)\right | du. \] 
\\
Now, by the change of variable $wu-k=y$, and recalling that $\psi\in L^1(\mathbb{R})$, we have
\[ w\int_{\left | wu-k \right| < \frac{\delta}{2}w}\  \left |\psi(wu-k)\right | du\le \int_{\mathbb{R}} \left |\psi(u)\right | du=\lVert \psi \rVert_1,\] for every $w\ge 0$.
\\
Thus, \[ I_{1,1}<M_{0}(\varphi)\lVert \psi \rVert_1\varepsilon,\text{ for every $w\ge 0$}.\] 
Moreover, \[I_{1,2} \le2\lVert f \rVert_\infty\sum_{\left | wx-k \right| \le \frac{\delta}{2}w}  \left | \varphi(wx-k) \right | w \int_{\left | wu-k \right| > \frac{\delta}{2}w}\  \left |\psi(wu-k)\right |du, \]
with \[\int_{\left | wu-k \right| > \frac{\delta}{2}w}\  w\left |\psi(wu-k)\right | du=\int_{\left | y \right| > \frac{\delta}{2}w}\left | \psi(y) \right | \,dy\rightarrow 0,\,\,\text{as $w\rightarrow+\infty$},\]since $\psi\in L^1({\mathbb{R}})$. Then there exists $\bar{w_1}\in \mathbb{R}$ such that  \[ I_{1,2} \le 2\lVert f \rVert_\infty M_{0}(\varphi)\varepsilon,\text{ for every $w\ge\bar{w_1}$}.\]\\
By similar reasoning, we obtain the following inequality \[ I_2 \le 2\lVert f \rVert_\infty \lVert \psi \rVert_1\sum_{\left | wx-k \right| > \frac{\delta}{2}w} \left | \varphi(wx-k) \right |.\]\\
From property (ii) of Lemma \ref{Lemma3}, there exists $\bar{w_2}\in \mathbb{R}$ such that \[I_2\le 2\lVert f \rVert_\infty \lVert \psi \rVert_1\varepsilon,\text{ for every $w \ge \bar{w_2}$}.\]\\
Setting $K:=M_{0}(\varphi)\lVert \psi \rVert_1+2\lVert f \rVert_\infty(M_{0}(\varphi)+\lVert \psi \rVert_1)$ and $\bar{w}:=\max\left\{\bar{w_1},\bar{w_2}\right\}$, we have \[\left |( S_w^{\varphi,\psi}f)(x)-f(x)\right | \le K\varepsilon,\text{ for every $w\ge \bar{w}$}.\]\\
Finally, observing that the above estimate does not depend on $x\in\mathbb{R}$, we easily obtain \[\sup_{x\in\mathbb{R}}\left |( S_w^{\varphi,\psi}f)(x)-f(x)\right |=\lVert S_w^{\varphi,\psi}f-f\rVert_\infty \le K\varepsilon,\text{ for every $w\ge \bar{w}$} \] and thus the proof follows by the arbitrariness of $\varepsilon>0$.
\end{proof} 

In order to obtain a modular convergence theorem in $L^{\eta}(\mathbb{R})$, we will study a modular continuity property for the family of Durrmeyer operators $(S_w^{\varphi,\psi})_{w>0}$. From now on, we denote by $\eta$ a convex $\varphi$-function.
Now, we can prove the following.
\begin{teorema} \label{continuitamodulare} Let $\psi$ be a kernel such that $M_0(\psi)<+\infty$, and $f\in L^{\eta}(\mathbb{R})$ be fixed. Then there exists $\lambda>0$ such that
\[I^{\eta}[\lambda S_w^{\varphi,\psi}f ]\le\frac{M_0(\psi)\lVert \varphi \rVert_1}{M_0(\varphi)\tilde{M_0}(\psi)} I^{\eta}[\lambda M_0(\varphi)\tilde{M_0}(\psi)f],\,\,w>0.\]
In particular, $S_w^{\varphi,\psi}f$ is well-defined and belongs to $L^{\eta}(\mathbb{R})$, for every $w>0$.
\end{teorema}
\begin{proof} Since $f\in L^{\eta}(\mathbb{R})$, there exists $\overline{\lambda}>0$ such that $I^{\eta}[\overline{\lambda}f]<+\infty$. We consider now $\lambda>0$ such that \[\lambda M_0(\varphi)\tilde{M_0}(\psi)\le \overline{\lambda}.\]
Then we have $I^{\eta}[\lambda M_0(\varphi)\tilde{M_0}(\psi)f]<+\infty$. Applying Jensen inequality twice, the change of variable $wu-k=t$ and Fubini-Tonelli theorem, we obtain
\[I^{\eta}[\lambda \left (S_w^{\varphi,\psi}f\right )]=\int_{\mathbb{R}}\eta\left (\lambda \left | \left (S_w^{\varphi,\psi}f\right )(x) \right | \right )dx\]
\[=\int_{\mathbb{R}}\eta\left (\lambda \left | \sum_{k\in\mathbb{Z}}\varphi(wx-k)\left [ w\int_{\mathbb{R}}\psi(wu-k)f(u)du \right] \right | \right )dx\]
\[\le \int_{\mathbb{R}}\eta\left (\lambda  \sum_{k\in\mathbb{Z}}\left |\varphi(wx-k)\right|\left [ w\int_{\mathbb{R}}\left|\psi(wu-k)\right|\left|f(u)\right|du \right] \right )dx\]
\[\le\frac{1}{M_0(\varphi)}\int_{\mathbb{R}}\sum_{k\in\mathbb{Z}}\eta\left (\lambda M_0(\varphi) \lVert \psi \rVert_1 w \int_{\mathbb{R}}\frac{\left |\psi(wu-k)\right |}{ \lVert \psi \rVert_1}\left |f(u)\right |du \right )\left | \varphi(wx-k) \right | dx\]
\[=\frac{1}{M_0(\varphi)}\int_{\mathbb{R}}\sum_{k\in\mathbb{Z}}\eta\left (\lambda M_0(\varphi) \lVert \psi \rVert_1  \int_{\mathbb{R}}\frac{\left |\psi(t)\right |}{ \lVert \psi \rVert_1}\left |f\left(\frac{t+k}{w}\right)\right |dt \right )\left | \varphi(wx-k) \right | dx\]
\[\le\frac{1}{M_0(\varphi)\lVert \psi \rVert_1}\int_{\mathbb{R}}\sum_{k\in\mathbb{Z}}\left | \varphi(wx-k) \right | dx \, \int_{\mathbb{R}}\left |\psi(t)\right | \eta\left (\lambda M_0(\varphi)\lVert \psi \rVert_1 \left |f\left(\frac{t+k}{w}\right)\right |\right )dt \]
\[=\frac{1}{M_0(\varphi)\lVert \psi \rVert_1}\int_{\mathbb{R}}\sum_{k\in\mathbb{Z}}\left | \varphi(wx-k) \right | dx \, w\int_{\mathbb{R}}\left |\psi(wu-k)\right | \eta\left (\lambda M_0(\varphi)\lVert \psi \rVert_1 \left |f\left(u\right)\right |\right )du \]
\[=\frac{1}{M_0(\varphi)\lVert \psi \rVert_1}\int_{\mathbb{R}}\left | \varphi(y) \right | dy \, \int_{\mathbb{R}}\sum_{k\in\mathbb{Z}}\left |\psi(wu-k)\right | \eta\left (\lambda M_0(\varphi) \lVert \psi \rVert_1\left |f(u)\right |\right )du \]
\[\le\frac{M_0(\psi)}{M_0(\varphi)\lVert \psi \rVert_1}\int_{\mathbb{R}}\left | \varphi(y) \right | dy \, \int_{\mathbb{R}} \eta\left (\lambda M_0(\varphi)\lVert \psi \rVert_1 \left |f(u)\right |\right )du \]
\[=\frac{M_0(\psi)}{M_0(\varphi)\lVert \psi \rVert_1}\lVert \varphi \rVert_1 I^{\eta}[\lambda M_0(\varphi)\lVert \psi \rVert_1f]\]
\[=\frac{M_0(\psi)}{M_0(\varphi)\tilde{M_0}(\psi)}\lVert \varphi \rVert_1 I^{\eta}[\lambda M_0(\varphi)\tilde{M_0}(\psi)f]<+\infty,\]
with the change of variable $wx-k=y$.
\end{proof}
As a consequence of previous theorem, it turns out that the operators $S_w^{\varphi,\psi}$ are well-defined in $ L^{\eta}(\mathbb{R})$ and map $ L^{\eta}(\mathbb{R})$ into itself. Moreover, we also have that $S_w^{\varphi,\psi}$ is modularly continuous, i.e., for any modularly convergent sequence $\left(f_k\right)_k\subset L^{\eta}(\mathbb{R})$, with $f_k\rightarrow f\in L^{\eta}(\mathbb{R})$, it turns out that $I^{\eta}[\lambda \left (S_w^{\varphi,\psi}f-S_w^{\varphi,\psi}f_k\right )]\rightarrow 0,$ as $k\rightarrow+\infty$. Indeed, it is well-known that there exists $\overline{\lambda}>0$ such that $I^{\eta}[\overline{\lambda}(f-f_k)]\rightarrow0$, as $k\rightarrow+\infty$, and so, choosing $\lambda>0$ such that $\lambda M_0(\varphi)\tilde{M_0}(\psi)\le \overline{\lambda}$, we have:
\[I^{\eta}[\lambda \left (S_w^{\varphi,\psi}f-S_w^{\varphi,\psi}f_k\right )]=I^{\eta}[\lambda S_w^{\varphi,\psi}(f-f_k)]\]
\[\le\frac{M_0(\psi)\lVert \varphi \rVert_1}{M_0(\varphi)\tilde{M_0}(\psi)} I^{\eta}[\lambda M_0(\varphi)\tilde{M_0}(\psi)(f-f_k)]\]\[\le\frac{M_0(\psi)\lVert \varphi \rVert_1}{M_0(\varphi)\tilde{M_0}(\psi)} I^{\eta}[\overline{\lambda}(f-f_k)]\rightarrow 0,\]
as $k\rightarrow+\infty$.

  Now, we are able to prove the main theorem of this section.
\begin{teorema}\label{modularconvergence}Let $\psi$ be a kernel such that $M_0(\psi)<+\infty$, and let $f\in L^{\eta}(\mathbb{R})$ be fixed. Then there exists $\lambda>0$ such that
\[\lim_{w\rightarrow \infty}I^{\eta}[\lambda \left (S_w^{\varphi,\psi}f-f\right )]=0.\]
\end{teorema}
\begin{proof}
First of all, since $f\in L^\eta(\mathbb{R})$, we have that there exist $\lambda_1, \lambda_2>0$ such that $I^\eta[\lambda_1f]<+\infty$, and \[I^\eta[\lambda_2\left(f(\cdot)-f(\cdot+h)\right)]\rightarrow 0,\text{ as $h\rightarrow 0$},\] i.e., for every fixed $\varepsilon>0$ there exists $\delta>0$ such that \begin{equation}\label{modcont}\int_{\mathbb{R}}\eta\left(\lambda_2\left|f(u+h)-f(u)\right|\right)du<\varepsilon,\end{equation}
for every $h\in\mathbb{R}$ such that $\left| h\right|\le\delta$ (see, e.g., \cite{11}).\\ Now, we fix $\lambda>0$ such that
\[\lambda\le\min\left \{\frac{\lambda_2}{2M_0(\varphi)\lVert\psi\rVert_1},\frac{\lambda_1}{4M_0(\varphi)\lVert\psi\rVert_1}\right\}.\]
Thus, by the properties of the convex modular functional $I^\eta$, we can write what follows:
\[I^\eta\left[\lambda\left(S_w^{\varphi,\psi}f-f\right)\right]=\int_{\mathbb{R}}\eta\left(\lambda\left|\left(S_w^{\varphi,\psi}f\right)(x)-f(x)\right|\right)dx\]
\[=\int_{\mathbb{R}}\eta\left(\lambda\left|\left(S_w^{\varphi,\psi}f\right)(x)-\sum_{k\in\mathbb{Z}}\varphi(wx-k)w\int_{\mathbb{R}}\psi(wu-k)f\left(u+x-\frac{k}{w}\right)du\right.\right.\]\[\left.\left.+\sum_{k\in\mathbb{Z}}\varphi(wx-k)w\int_{\mathbb{R}}\psi(wu-k)f\left(u+x-\frac{k}{w}\right)du-f(x)\right|\right)dx\]
\[\le\frac{1}{2}\left \{ \int_{\mathbb{R}}\eta\left(2\lambda\left|\left(S_w^{\varphi,\psi}f\right)(x)-\sum_{k\in\mathbb{Z}}\varphi(wx-k)w\int_{\mathbb{R}}\psi(wu-k)f\left(u+x-\frac{k}{w}\right)du\right|\right)dx\right.\]
\[\left.+\int_{\mathbb{R}}\eta\left(2\lambda\left|\sum_{k\in\mathbb{Z}}\varphi(wx-k)w\int_{\mathbb{R}}\psi(wu-k)f\left(u+x-\frac{k}{w}\right)du-f(x)\right|\right)dx\right \}\]
\[=:\frac{1}{2}\left\{J_1+J_2\right \}.\]
First, we estimate $J_1$. Applying Jensen inequality twice similarly to the proof of Theorem \ref{continuitamodulare}, the change of variable $wx-k=t$ and Fubini-Tonelli theorem, we obtain
\[\left|J_1\right|\le\int_{\mathbb{R}}\eta\left(2\lambda\sum_{k\in\mathbb{Z}}\left|\varphi(wx-k)\right|w\int_{\mathbb{R}}\left|\psi(wu-k)\right|\left|f\left(u+x-\frac{k}{w}\right)-f(u)\right|du\,\right)dx\]
\[\le\frac{1}{M_0(\varphi)w}\int_{\mathbb{R}}\left|\varphi(t)\right|\left[\sum_{k\in\mathbb{Z}}\eta\left(2\lambda M_0(\varphi)w\int_{\mathbb{R}}\left|\psi(wu-k)\right|\left|f\left(u+\frac{t}{w}\right)-f(u)\right|du\right)\right]dt\]
\[\le\frac{1}{M_0(\varphi)\lVert\psi\rVert_1}\int_{\mathbb{R}}\left|\varphi(t)\right|\left[\sum_{k\in\mathbb{Z}}\int_{\mathbb{R}}\left|\psi(wu-k)\right|\eta\left(2\lambda M_0(\varphi)\lVert\psi\rVert_1\left|f\left(u+\frac{t}{w}\right)-f(u)\right|\right)du\right]dt\]
\[\le\frac{1}{M_0(\varphi)\lVert\psi\rVert_1}\int_{\mathbb{R}}\left|\varphi(t)\right|\left[\int_{\mathbb{R}}M_0(\psi) \eta\left(2\lambda M_0(\varphi)\lVert\psi\rVert_1\left|f\left(u+\frac{t}{w}\right)-f(u)\right|\right)du\right]dt.\]
Now, using $\delta$ given in (\ref{modcont}), we can split the above integral as follows
\[\left|J_1\right|\le\frac{M_0(\psi)}{M_0(\varphi)\lVert\psi\rVert_1}\left\{\int_{\left|t\right|\le\delta w}+\int_{\left|t\right|>\delta w}    \right\}\left|\varphi(t)\right|\left[\int_{\mathbb{R}}\eta\left(2\lambda M_0(\varphi)\lVert\psi\rVert_1\left|f\left(u+\frac{t}{w}\right)-f(u)\right|\right)du\right]dt\]
\[=:J_{1,1}+J_{1,2}.\]
Now, using the inequality in (\ref{modcont}) with $h=\frac{t}{w}$, we have
\[\left|J_{1,1}\right|=\frac{M_0(\psi)}{M_0(\varphi)\lVert\psi\rVert_1}\int_{\left|t\right|\le\delta w}\left|\varphi(t)\right|\left[\int_{\mathbb{R}}\eta\left(2\lambda M_0(\varphi)\lVert\psi\rVert_1\left|f\left(u+\frac{t}{w}\right)-f(u)\right|\right)du\right]dt\]
\[\le\frac{M_0(\psi)}{M_0(\varphi)\lVert\psi\rVert_1}\int_{\left|t\right|\le\delta w}\left|\varphi(t)\right|\left[\int_{\mathbb{R}}\eta\left(\lambda_2\left|f\left(u+\frac{t}{w}\right)-f(u)\right|\right)du\right]dt\le\frac{M_0(\psi)\lVert\varphi\rVert_1}{M_0(\varphi)\lVert\psi\rVert_1}\varepsilon,\]
for every $w>0$.\\
For what concerns $J_{1,2}$, by the convexity of $\eta$, we have
\[\left|J_{1,2}\right|\le\frac{M_0(\psi)}{M_0(\varphi)\lVert\psi\rVert_1}\int_{\left|t\right|>\delta w}\left|\varphi(t)\right|\frac{1}{2}\left[\int_{\mathbb{R}}\eta\left(4\lambda M_0(\varphi)\lVert\psi\rVert_1\left|f\left(u+\frac{t}{w}\right)\right|\right)du+\right.\]
\[\left.+\int_{\mathbb{R}}\eta\left(4\lambda M_0(\varphi)\lVert\psi\rVert_1\left|f\left(u\right)\right|\right)du\right]dt.\]
Now we can observe that, since $\varphi\in L^1(\mathbb{R})$, there exists $\overline{w_1}>0$ such that \[\int_{\left|t\right|>\delta w}\left|\varphi(t)\right|dt<\varepsilon,\] for every $w\ge\overline{w_1}$. Moreover, noting that 
\[\int_{\mathbb{R}}\eta\left(4\lambda M_0(\varphi)\lVert\psi\rVert_1\left|f\left(u+\frac{t}{w}\right)\right|\right)du=\int_{\mathbb{R}}\eta\left(4\lambda M_0(\varphi)\lVert\psi\rVert_1\left|f\left(u\right)\right|\right)du,\]
for every $t\in\mathbb{R}$ and $w>0$, we have
\[\left|J_{1,2}\right|\le\frac{M_0(\psi)}{M_0(\varphi)\lVert\psi\rVert_1}\int_{\left|t\right|>\delta w}\left|\varphi(t)\right|dt\int_{\mathbb{R}}\eta\left(4\lambda M_0(\varphi)\lVert\psi\rVert_1\left|f\left(u\right)\right|\right)du\]
\[\le\frac{M_0(\psi)}{M_0(\varphi)\lVert\psi\rVert_1} I^\eta\left[\lambda_1f\right]\int_{\left|t\right|>\delta w}\left|\varphi(t)\right|dt<\frac{M_0(\psi)}{M_0(\varphi)\lVert\psi\rVert_1}I^\eta\left[\lambda_1f\right]\varepsilon, \]
for every $w\ge\overline{w_1}$.\\
Now, we estimate $J_2$. By the change of variable $t = u-\frac{k}{w}$, applying Jensen inequality twice and Fubini-Tonelli theorem, we have
\[\left|J_2\right|=\int_{\mathbb{R}}\eta\left(2\lambda\left|\sum_{k\in\mathbb{Z}}\varphi(wx-k)w\int_{\mathbb{R}}\psi(wu-k)\left[f\left(u+x-\frac{k}{w}\right)-f(x)\right]du\right|\right)dx\]
\[=\int_{\mathbb{R}}\eta\left(2\lambda\left|\sum_{k\in\mathbb{Z}}\varphi(wx-k)w\int_{\mathbb{R}}\psi(wt)\left[f\left(t+x\right)-f(x)\right]dt\right|\right)dx\]
\[\le\int_{\mathbb{R}}\eta\left(2\lambda\sum_{k\in\mathbb{Z}}\left|\varphi(wx-k)\right|w\int_{\mathbb{R}}\left|\psi(wt)\right|\left|f\left(t+x\right)-f(x)\right|dt\right)dx\]
\[\le\frac{1}{M_0(\varphi)}\int_{\mathbb{R}}\sum_{k\in\mathbb{Z}}\left|\varphi(wx-k)\right|\eta\left(2\lambda M_0(\varphi) w\int_{\mathbb{R}}\left|\psi(wt)\right|\left|f\left(t+x\right)-f(x)\right|dt\right)dx\]
\[\le\frac{1}{M_0(\varphi)}\int_{\mathbb{R}}M_0(\varphi)\eta\left(2\lambda M_0(\varphi) \int_{\mathbb{R}}\left|\psi(y)\right|\left|f\left(\frac{y}{w}+x\right)-f(x)\right|dy\right)dx\]
\[\le\frac{1}{\lVert\psi\rVert_1}\int_{\mathbb{R}}\left[ \int_{\mathbb{R}}\left|\psi(y)\right|\eta\left(2\lambda M_0(\varphi)\lVert\psi\rVert_1\left|f\left(\frac{y}{w}+x\right)-f(x)\right|\right)dy\right]dx\]
\[=\frac{1}{\lVert\psi\rVert_1}\int_{\mathbb{R}}\left|\psi(y)\right|\left[\int_{\mathbb{R}}\eta\left(2\lambda M_0(\varphi)\lVert\psi\rVert_1\left|f\left(\frac{y}{w}+x\right)-f(x)\right|\right)dx\right]dy,\]
where we have used the change of variable $y=wt$. Then, using again $\delta$ given in (\ref{modcont}), we can rewrite the above integral as follows:
\[\left|J_2\right|\le\frac{1}{\lVert\psi\rVert_1}\left\{\int_{\left|y\right|\le\delta w}+\int_{\left|y\right|>\delta w} \right\}\left|\psi(y)\right|\left[ \int_{\mathbb{R}}\eta\left(2\lambda M_0(\varphi)\lVert\psi\rVert_1\left|f\left(\frac{y}{w}+x\right)-f(x)\right|\right)dx\right]dy\]
\[=:J_{2,1}+J_{2,2}.\]
Thus, similarly to before, using the inequality in (\ref{modcont}) with $h=\frac{y}{w}$, we obtain
\[\left|J_{2,1}\right|=\frac{1}{\lVert\psi\rVert_1}\int_{\left|y\right|\le\delta w}\left|\psi(y)\right|\left[\int_{\mathbb{R}}\eta\left(2\lambda M_0(\varphi) \lVert\psi\rVert_1\left|f\left(\frac{y}{w}+x\right)-f(x)\right|\right)dx\right]dy\]
\[\le\frac{1}{\lVert\psi\rVert_1}\int_{\left|y\right|\le\delta w}\left|\psi(y)\right|\left[\int_{\mathbb{R}}\eta\left(\lambda_2\left|f\left(\frac{y}{w}+x\right)-f(x)\right|\right)dx\right]dy<\frac{1}{\lVert\psi\rVert_1}\lVert\psi\rVert_1\varepsilon=\varepsilon,\]
for every $w>0$.\\
Now, for what concerns the last term $J_{2,2}$, by the convexity of $\eta$, we have
\[\left|J_{2,2}\right|\le\frac{1}{\lVert\psi\rVert_1}\int_{\left|y\right|>\delta w}\left|\psi(y)\right|\frac{1}{2}\left[\int_{\mathbb{R}}\eta\left(4\lambda M_0(\varphi)\lVert\psi\rVert_1\left|f\left(x+\frac{y}{w}\right)\right|\right)dx\right.\]\[\left.+\int_{\mathbb{R}}\eta\left(4\lambda M_0(\varphi)\lVert\psi\rVert_1\left|f\left(x\right)\right|\right)dx\right]dy.\]
Now, we can observe that, since $\psi\in L^1(\mathbb{R})$, there exists $\overline{w_2}>0$ such that \[\int_{\left|y\right|>\delta w}\left|\psi(y)\right|dy<\varepsilon,\] for every $w\ge\overline{w_2}$ and, similarly to before, we have
\[\left|J_{2,2}\right|\le\frac{1}{\lVert\psi\rVert_1}\int_{\mathbb{R}}\eta\left(4\lambda M_0(\varphi) \lVert\psi\rVert_1\left|f(x)\right|\right)dx\int_{\left|y\right|>\delta w}\left|\psi(y)\right|dy\]
\[\le\frac{1}{\lVert\psi\rVert_1}\int_{\left|y\right|>\delta w}\left|\psi(y)\right|dy\, I^\eta[\lambda_1f]<\frac{\varepsilon}{\lVert\psi\rVert_1}I^\eta[\lambda_1f],\]
for every $w\ge\overline{w_2}$.\\
Finally, setting $\overline{w}:=\max\left\{ \overline{w_1},\overline{w_2}\right\}$ and \[K:=\frac{M_0(\psi)}{2\lVert\psi\rVert_1}\left[\frac{\lVert\varphi\rVert_1}{M_0(\varphi)}+\frac{\lVert\psi\rVert_1}{M_0(\psi)}+I^\eta[\lambda_1f]\left(\frac{1}{M_0(\varphi)}+\frac{1}{M_0(\psi)}\right)\right],\]
we have
\[I^\eta\left[\lambda\left(S_w^{\varphi,\psi}f-f\right)\right]\le K\varepsilon,\]
for every $w\ge\overline{w}$. Thus, the proof follows by the arbitrariness of $\varepsilon$.
\end{proof}

\section{Quantitative estimates}
Here we provide a quantitative estimate for the rate of convergence of the Durrmeyer sampling operators for $f\in C(\mathbb{R})$, in terms of the modulus of continuity, defined by
\[\omega(f,\delta):=\sup\left \{\left|f(x)-f(y)\right|: \left |x-y\right|<\delta,\,x, y\in \mathbb{R} \right \},\]
$\delta>0$.\\ We recall the following well-known inequality\begin{equation}\label{inequality}
\omega(f,\lambda\delta)\le(\lambda+1)\omega(f,\delta),\,\,\,\text{for every $\delta,\lambda>0$.}\end{equation}

We can prove what follows.
 \begin{teorema}\label{Quantitativeestimate} Suppose that $\varphi$ and $\psi$ are such that $M_1(\varphi)+\tilde {M}_1(\psi)<+\infty$ and let $f\in C(\mathbb{R})$. Then we have \[\lVert S_w^{\varphi,\psi}f-f\rVert_\infty\le C^{\varphi,\psi}\omega\left(f,\frac{1}{w}\right),\]for every $w>0$, where $C^{\varphi,\psi}=M_0(\varphi)\left(\tilde {M}_0(\psi)+\tilde {M}_1(\psi)\right)+M_1(\varphi)\tilde {M}_0(\psi)$.
 \end{teorema} 
 \begin{proof}Let $x\in\mathbb{R}$ be fixed. We have
\[\left |( S_w^{\varphi,\psi}f)(x)-f(x)\right |\le\sum_{k\in\mathbb{Z}}\left | \varphi(wx-k) \right | w \int_{\mathbb{R}} \left |\psi(wu-k)\right | \left | f(u)-f(x) \right | du \]
\[\le\sum_{k\in\mathbb{Z}}\left | \varphi(wx-k) \right | w \int_{\mathbb{R}} \left |\psi(wu-k)\right | \omega\left(f,\left|u-x\right|\right)du \]
\[\le\sum_{k\in\mathbb{Z}}\left | \varphi(wx-k) \right | w \int_{\mathbb{R}} \left |\psi(wu-k)\right |\omega\left(f,\frac{1}{w}\right)(1+w\left|u-x\right|) du\]
\[=\omega\left(f,\frac{1}{w}\right)\sum_{k\in\mathbb{Z}}\left | \varphi(wx-k) \right | w \int_{\mathbb{R}} \left |\psi(wu-k)\right |(1+w\left|u-x\right|) du\]
\[=\omega\left(f,\frac{1}{w}\right)\sum_{k\in\mathbb{Z}}\left | \varphi(wx-k) \right | \left \{w \int_{\mathbb{R}} \left |\psi(wu-k)\right |du\right.\]
\[\left.+w \int_{\mathbb{R}} \left |\psi(wu-k)\right| w \left|u-x\right|du\right \},\]
for every $w>0$, where the previous estimate is a consequence of the definition of $\omega(f,\frac{1}{w})$ and of (\ref{inequality}) with $\lambda= w\left|u-x\right|>0$ and $\delta=\frac{1}{w}$.
Now we estimate the following term:
\[w \int_{\mathbb{R}} \left |\psi(wu-k)\right| \left|wu-wx\right|du\]
\[\le\int_{\mathbb{R}} w \left |\psi(wu-k)\right| \left|wu-k\right|du+\int_{\mathbb{R}} w \left |\psi(wu-k)\right| \left|k-wx\right|du\]
\[=\tilde {M}_1(\psi)+\left|k-wx\right|\tilde {M}_0(\psi).\]
So we obtain
\[\left |( S_w^{\varphi,\psi}f)(x)-f(x)\right |\le\omega\left(f,\frac{1}{w}\right)\sum_{k\in\mathbb{Z}}\left | \varphi(wx-k) \right | \left \{\tilde {M}_0(\psi)+\tilde {M}_1(\psi)+\left|k-wx\right|\tilde {M}_0(\psi)\right \}\]
\[=\omega\left(f,\frac{1}{w}\right)\left \{\sum_{k\in\mathbb{Z}}\left|\varphi(wx-k)\right|\tilde {M}_0(\psi)+\sum_{k\in\mathbb{Z}}\left|\varphi(wx-k)\right|\tilde {M}_1(\psi)\right.\]
\[\left.+\sum_{k\in\mathbb{Z}}\left|\varphi(wx-k)\right| \left|k-wx\right|\tilde {M}_0(\psi)\right \}\]
\[=\omega\left(f,\frac{1}{w}\right)\left \{M_0(\varphi)\left(\tilde {M}_0(\psi)+\tilde {M}_1(\psi)\right)+M_1(\varphi)\tilde {M}_0(\psi)\right \}.\]
Now, setting $C^{\varphi,\psi}:=M_0(\varphi)\left(\tilde {M}_0(\psi)+\tilde {M}_1(\psi)\right)+M_1(\varphi)\tilde {M}_0(\psi)<+\infty$, we finally obtain
\[\lVert S_w^{\varphi,\psi}f-f\rVert_\infty\le C^{\varphi,\psi}\omega\left(f,\frac{1}{w}\right),\]for every $w>0$. This completes the proof.
\end{proof} 
Recalling the definition of Lipschitz classes $Lip\,\alpha$, namely,
\[Lip\,\alpha:=\left \{f\in C(\mathbb{R}):\omega(f,\delta)=\mathcal{O}(\delta^\alpha),\text{ as }\delta\rightarrow 0^+ \right \},\]
with $0<\alpha\le 1$, by Theorem \ref{Quantitativeestimate}, we can immediately obtain the following corollary.
 \begin{cor}Under the assumptions of Theorem \ref{Quantitativeestimate}, and assuming in addition that $f\in Lip\,\alpha$, $0<\alpha\le1$, then \[\lVert S_w^{\varphi,\psi}f-f\rVert_\infty\le Cw^{-\alpha},\]for every $w>0$,
and for a suitable positive constant $C>0$.
 \end{cor}
 
\section{Applications to particular cases}
In this section we want to show how the Durrmeyer sampling type series generalize some other well-known families of sampling type series. Moreover, we will also consider applications to some special instances of Orlicz spaces.\\
\\
In order to show that the generalized sampling type series, introduced by Butzer in the 80s  (see, e.g., \cite{19, 17, 16, 18, 2001v, 2018aasfm,2020bjma}), are particular cases of the Durrmeyer sampling type series, we need to give a distributional interpretation of the above operators, choosing, e.g., as kernel $\psi$ the Dirac delta distribution $\delta$.
Indeed, using the scaling and convolution property of the Dirac delta distribution, and recalling that $\delta$ is even, the generalized sampling operators
\[(G_w^\varphi f)(x):=\sum_{k\in\mathbb{Z}}f\left( \frac{k}{w}\right)\varphi(wx-k), \qquad x\in \mathbb{R},\]
can be obtained as follows
\[\left (S_w^{\varphi,\delta}f\right)(x)= \sum_{k\in\mathbb{Z}} \varphi(wx-k) w \int_{\mathbb{R}}\delta(wu-k)f(u)du\]
\[= \sum_{k\in\mathbb{Z}} \varphi(wx-k) w \int_{\mathbb{R}}\delta\left [w\left(u-\frac{k}{w}\right)\right]f(u)du\]
\[=\sum_{k\in\mathbb{Z}} \varphi(wx-k)\int_{\mathbb{R}}\delta\left (u-\frac{k}{w}\right)f(u)du\]
\[=\sum_{k\in\mathbb{Z}} \varphi(wx-k)(\delta*f)\left (\frac{k}{w}\right )\]\[=\sum_{k\in\mathbb{Z}} \varphi(wx-k)f\left( \frac{k}{w}\right)= (G_w^\varphi f)(x),\]
for any $f\in C(\mathbb{R})$. Thus, $\left (S_w^{\varphi,\delta}f\right)(x)=(G_w^\varphi f)(x)$, for every $x\in\mathbb{R}$ and $w>0$.\\
\\
Similarly to what has been made for the generalized sampling operators, also the sampling Kantorovich operators (see, e.g., \cite{1})
\[(K_w^\chi f)(x)=\sum_{k\in\mathbb{Z}}\chi(wx-k)\left [w\int_{\frac{k}{w}}^{\frac{k+1}{w}}f(u)du\right],\text{$\qquad x\in\mathbb{R}$,}\]
can be viewed as Durrmeyer sampling type operators.\\
Indeed, for $f\in L^\infty(\mathbb{R})$ and $\psi(t)=\chi_{[0,1]}(t)$, $t\in\mathbb{R}$, where $\chi$ is the characteristic function of the set $[0,1]\subset\mathbb{R}$, we have
\[\left ( S_w^{\varphi,\chi_{[0,1]}}f\right)(x)=\sum_{k\in\mathbb{Z}} \varphi(wx-k) w \int_{\mathbb{R}}\chi_{[0,1]}(wu-k)f(u)du\]
\[=\sum_{k\in\mathbb{Z}} \varphi(wx-k) w \int_{\frac{k}{w}}^{\frac{k+1}{w}}f(u)du=(K_w^\varphi f)(x).\]
Thus, $\left (S_w^{\varphi,\chi_{[0,1]}}f\right)(x)=(K_w^\varphi f)(x)$, for every $x\in\mathbb{R}$ and $w>0$.\\
Finally, we observe that, in this case, $\psi$ satisfies the condition of Remark \ref{osservazione3.1} (c). Hence all the modular convergence results hold.\\ For further theoretical results concerning sampling Kantorovich operators, the readers can see the following references \cite{2011umi,2014jmaa, ORTA1, 2017jmaa,2019amp,2019dol,2019pems, 2019jat}; moreover, for applications to image reconstruction and enhancement, see, e.g., \cite{2020amc, 2020mmas}.\\
\\
Finally we will apply the previous convergence results in some useful cases of Orlicz spaces.\\
First we consider the particular case when $\eta(u)=u^p$ for $u\ge 0$ and $1\le p<+\infty$. Here $L^\eta(\mathbb{R})=E^\eta(\mathbb{R})=L^p(\mathbb{R})$, $1\le p<+\infty$ and in this frame, the modular convergence and the usual Luxemburg norm-convergence are equivalent. From the theory developed in the previous sections, we have the following corollaries.
\begin{cor}Let $\psi$ be such that $M_0(\psi)<+\infty$. Then, for every $f\in L^p(\mathbb{R})$, $1\le p<+\infty$, we have
\[\lVert S_w^{\varphi,\psi}f\rVert_p\le M_0(\psi)^{\frac{1}{p}}M_0(\varphi)^{\frac{p-1}{p}}\lVert\varphi\rVert_1^{1/p}\tilde{M_0}(\psi)^{\frac{p-1}{p}}\lVert f\rVert_p,\,\,\,w>0.\]
In particular, $S_w^{\varphi,\psi}f$ is well-defined in $L^p(\mathbb{R})$ and $S_w^{\varphi,\psi}f\in L^p(\mathbb{R})$ whenever $f\in L^p(\mathbb{R})$.
\end{cor}
\begin{proof}A direct application of Theorem \ref{continuitamodulare} with $\eta(u)=u^p$, yields
\[\lVert S_w^{\varphi,\psi}f\rVert_p^p\le\frac{M_0(\psi)\lVert\varphi\rVert_1}{M_0(\varphi)\tilde{M_0}(\psi)}M_0(\varphi)^p\tilde{M_0}(\psi)^p\lVert f\rVert_p^p,\]
from which the assertion follows.
\end{proof}
Moreover we immediately obtain the following convergence result.
\begin{cor}\label{lp}Let $\psi$ be such that $M_0(\psi)<+\infty$. For every $f\in L^p(\mathbb{R})$, $1\le p<+\infty$, we have
\[\lim_{w\rightarrow+\infty}\lVert S_w^{\varphi,\psi}f-f\rVert_p=0.\]
\end{cor}
As another important case, we can consider the function $\eta_{\alpha,\beta}(u)=u^\alpha\log^\beta(e+u)$, $u\ge0$ for $\alpha\ge 1$ and $\beta>0$. The corresponding Orlicz spaces are the so-called interpolation spaces and are given by the set of functions $f\in\mathcal{M}(\mathbb{R})$ for which
\[I^{\eta_{\alpha,\beta}}[\lambda f]=\int_{\mathbb{R}}\left (\lambda\left |f(x)\right|  \right)^\alpha\log^\beta(e+\lambda\left |f(x)\right| )dx<+\infty,\]
for some $\lambda>0$, and they are denoted by $L^\alpha\log^\beta L(\mathbb{R})$. Note that the function $\eta_{\alpha,\beta}$ satisfies the $\Delta_2$-property, which means that $L^\alpha\log^\beta L(\mathbb{R})$ coincides with the space of its finite elements $E^{\eta_{\alpha,\beta}}(\mathbb{R})$.
As a consequence of the Theorem \ref{continuitamodulare}, we can obtain the following corollary, e.g. for the case $\alpha=\beta=1$.
\begin{cor}Let $\psi$ be such that $M_0(\psi)<+\infty$. For every $f\in L\log L$, there holds
\[\int_{\mathbb{R}}\left |S_w^{\varphi,\psi}f(x)\right |\log(e+\lambda\left |S_w^{\varphi,\psi}f(x)\right |)dx\]
\[\le M_0(\psi)\lVert\varphi\rVert_1\int_{\mathbb{R}}\left |f(u)\right |\log(e+\lambda M_0(\varphi)\tilde{M_0}(\psi)\left |f(u)\right |)du,\]
$\lambda>0$. In particular, $S_w^{\varphi,\psi}f$ is well-defined in $ L\log L$ and $S_w^{\varphi,\psi}f\in L\log L$ whenever $f\in L\log L$.
\end{cor}
Since in the above case of Orlicz spaces the $\Delta_2$-condition is fulfilled, the modular convergence and the norm convergence are equivalent and we immediately obtain the following convergence theorem.
\begin{cor}\label{log}Let $\psi$ be such that $M_0(\psi)<+\infty$. For every $f\in L\log L$ and for every $\lambda>0$, we have
\[\lim_{w\rightarrow+\infty}\int_{\mathbb{R}}\left |S_w^{\varphi,\psi}f(x)-f(x)\right |\log(e+\lambda\left |S_w^{\varphi,\psi}f(x)-f(x)\right |)dx=0,\]
or, equivalently,
\[\lim_{w\rightarrow+\infty}\lVert S_w^{\varphi,\psi}f-f\rVert_{L\log L}=0,\]
where $\lVert\cdot\rVert_{L\log L}$ is the Luxemburg norm associated to $I^{\eta_{1,1}}$.
\end{cor}
As last particular case, we consider the exponential spaces generated by the $\varphi$-function $\eta_\alpha(u)=e^{u^\alpha}-1$, $u\ge 0$ for some $\alpha>0$. Here the Orlicz space $L^{\eta_\alpha}(\mathbb{R})$ consists of those functions $f\in\mathcal{M}(\mathbb{R})$ for which\[I^{\eta_{\alpha}}[\lambda f]=\int_{\mathbb{R}}\left (\exp{\left(\lambda\left| f(x)\right|\right)^\alpha}-1\right)dx<+\infty\]
for some $\lambda>0$.
Since $\eta_\alpha$ does not satisfy the $\Delta_2$-property, the space $L^{\eta_\alpha}(\mathbb{R})$ does not coincide with the space of its finite elements $E^{\eta_\alpha}(\mathbb{R})$. As a consequence, modular convergence is no more equivalent to norm convergence. By Theorem \ref{continuitamodulare}, we can obtain the following.
\begin{cor}Let $\psi$ be such that $M_0(\psi)<+\infty$. For every $f\in L^{\eta_\alpha}(\mathbb{R})$, there holds
\[\int_{\mathbb{R}}\left (\exp{\left(\lambda\left| S_w^{\varphi,\psi}(x)\right|\right)^\alpha}-1\right)dx\]
\[\le \frac{M_0(\psi)\lVert\varphi\rVert_1}{M_0(\varphi)\tilde{M_0}(\psi)}\int_{\mathbb{R}}\left (\exp{\left(\lambda M_0(\varphi)\tilde{M_0}(\psi)\left| f(x)\right|\right)^\alpha}-1\right)dx,\]
$\lambda>0$. In particular, $S_w^{\varphi,\psi}f$ is well-defined in $L^{\eta_\alpha}(\mathbb{R})$ and $S_w^{\varphi,\psi}f\in L^{\eta_\alpha}(\mathbb{R})$ whenever $f\in L^{\eta_\alpha}(\mathbb{R})$.
\end{cor}
Since in this case $\Delta_2$-property is not fulfilled, we can only state a result on modular convergence rather than on norm convergence. The next corollary follows immediately from Theorem \ref{modularconvergence}.
\begin{cor}\label{exp}Let $\psi$ be such that $M_0(\psi)<+\infty$. For every $f\in L^{\eta_\alpha}(\mathbb{R})$, there exists $\lambda>0$ such that
\[\lim_{w\rightarrow \infty}\int_{\mathbb{R}}\left (\exp{\left(\lambda\left| S_w^{\varphi,\psi}f(x)-f(x)\right|\right)^\alpha}-1\right)dx=0.\]
\end{cor}
\section{Examples with graphical representations}
 In this last section we want to show specific examples of kernel functions $\varphi$ and $\psi$, for which the results proved in this paper hold, together with some graphical examples.
 \begin{enumerate}
 \item We put $\varphi(t)=\sigma_3(t)$, where $\sigma_3$ is the central B-spline of order 3, i.e,
\begin{equation}\label{s3}\sigma_3(t):=\frac{1}{2}\sum_{j=0}^{3}\binom{3}{j}\left(\frac{3}{2}+t-j\right)^2, \,\,\,t\in\mathbb{R},
\end{equation}
(see Figure \ref{graficosigma3}).
 \begin{figure}[h!]
    \begin{center}
    \includegraphics[width=6.6cm]{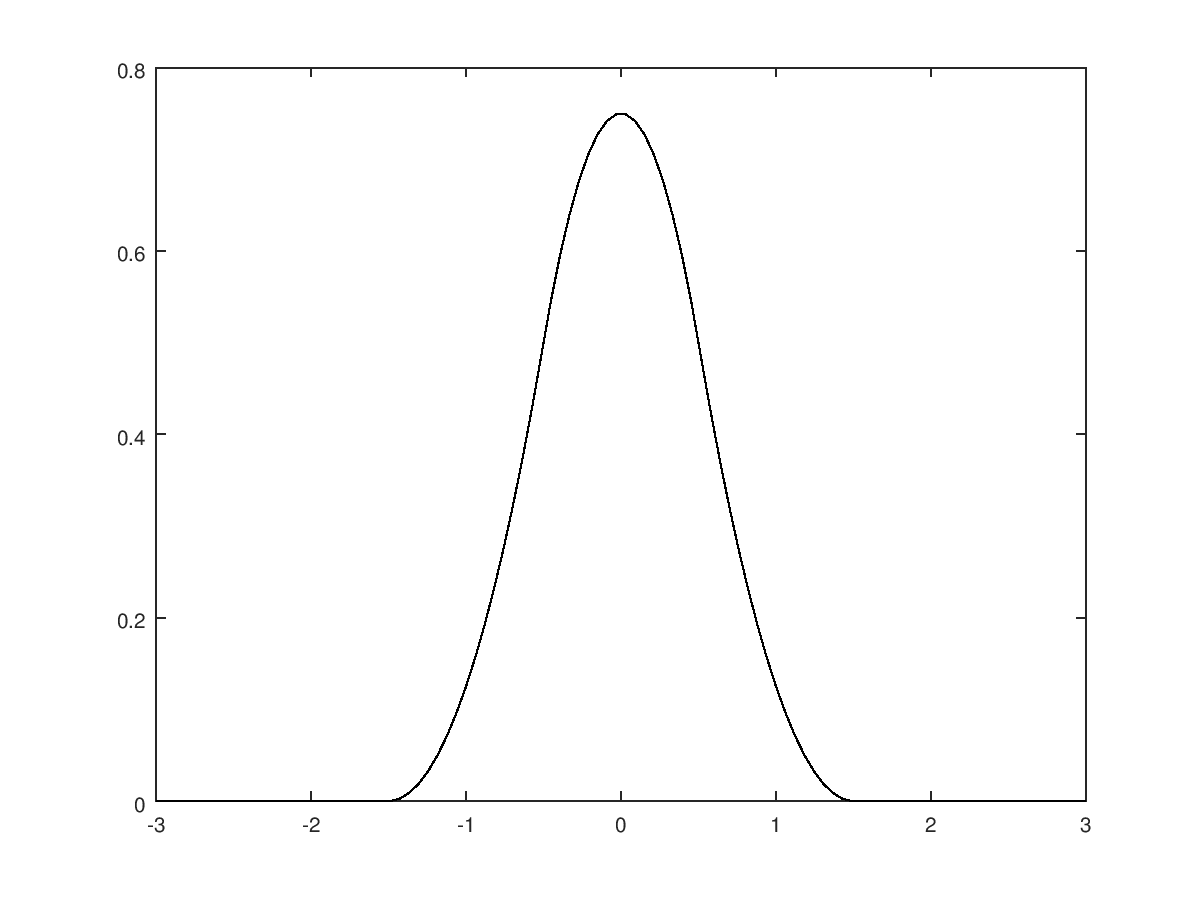}
    \end{center}
    \caption{The spline function $\sigma_3$.}
    \label{graficosigma3}
\end{figure}
\\In general, we define the central B-spline of order $n\in\mathbb{N}$ as 
\[\sigma_n(t):=\frac{1}{(n-1)!}\sum_{j=0}^{n}(-1)^j\binom{n}{j}\left(\frac{n}{2}+t-j \right)_+^{n-1},\,\, \,t\in\mathbb{R},\]
where $(\cdot)_+$ denotes the positive part, i.e., $(t)_+:=\max\{t,0\}, \,t\in\mathbb{R}$.

  The Fourier transform of $\sigma_n$ is given by
\[\widehat{\sigma_n}(v)=\text{sinc}^n\left(\frac{v}{2\pi}\right),\text{  }v\in\mathbb{R},\]
(see, e.g., \cite{1981jws,18,2020drna}), where the sinc-function is defined by
\[\text{sinc}(v):=
\begin{cases}
\frac{\sin\pi v}{\pi v}, & v\in\mathbb{R}\setminus 0, \\

1, & v=0.
\end{cases}
\]
 The functions $\sigma_n$ are bounded on $\mathbb{R}$ for all $n\in\mathbb{N}$ with compact support $[-n/2,n/2]$. This implies that $\sigma_n\in L^1(\mathbb{R})$ and the moment condition $M_r(\varphi)<+\infty$ is satisfied for all $r>0$. It is well-known that the singularity assumption (\ref{IpotesiD}) on $\varphi$ is equivalent to prove the following condition expressed in terms of $\widehat{\sigma_n}$:
\begin{equation}\label{poisson}\widehat{\sigma_n}(2k\pi)=\begin{cases}
1, & k=0, \\
0, & k\in\mathbb{Z}\setminus 0.
\end{cases}\end{equation}
The equivalence between the two conditions is a direct consequence of the Poisson summation formula (see, e.g., \cite{1971ap}).
\\
Rewriting explicitly the expression in (\ref{s3}), we have 
\[\sigma_3(t):=
\begin{cases}
\frac{3}{4}-t^2, & \left| t \right|\le \frac{1}{2}, \\
\frac{1}{2}\left(\frac{3}{2}-\left| t\right|\right)^2, & \frac{1}{2}<\left| t \right|\le \frac{3}{2}, \\
0, & \left| t \right|\ge\frac{3}{2},
\end{cases}
\, \,t\in\mathbb{R}.\]

  Now, we put $\psi(t)=\chi_{[0,1]}(t)$, $t\in\mathbb{R}$. In this case, Durrmeyer sampling type series becomes
\[\left (S_w^{\sigma_3,\chi_{[0,1]}}f\right)(x)=\sum_{k\in\mathbb{Z}}\sigma_3(wx-k) w \int_{\mathbb{R}}\chi_{[0,1]}(wu-k)f(u)du\]
\[=\sum_{k\in\mathbb{Z}}\sigma_3(wx-k)w\int_{\frac{k}{w}}^{\frac{k+1}{w}}f(u)du=\left(K_w^{\sigma_3}f\right)(x),\,\, \,x\in\mathbb{R}.\]
Next we apply the Durrmeyer sampling operator previously obtained, $S_w^{\sigma_3,\chi_{[0,1]}}$, to a particular regular function in $C(\mathbb{R})$, namely $f(x):=\frac{1}{x^2+1},\,x\in\mathbb{R}.$\\
The Durrmeyer sampling series $S_w^{\sigma_3,\chi_{[0,1]}}f$ of this function for $w=5$ and $w=10$ are given in Figure \ref{n5}. The red dotted line is the graph of the operator $S_w^{\sigma_3,\chi_{[0,1]}}f$, while the black line is the graph of $f$.\\
\begin{figure*}[tbph]
\centering
 \includegraphics[width=6.6cm]{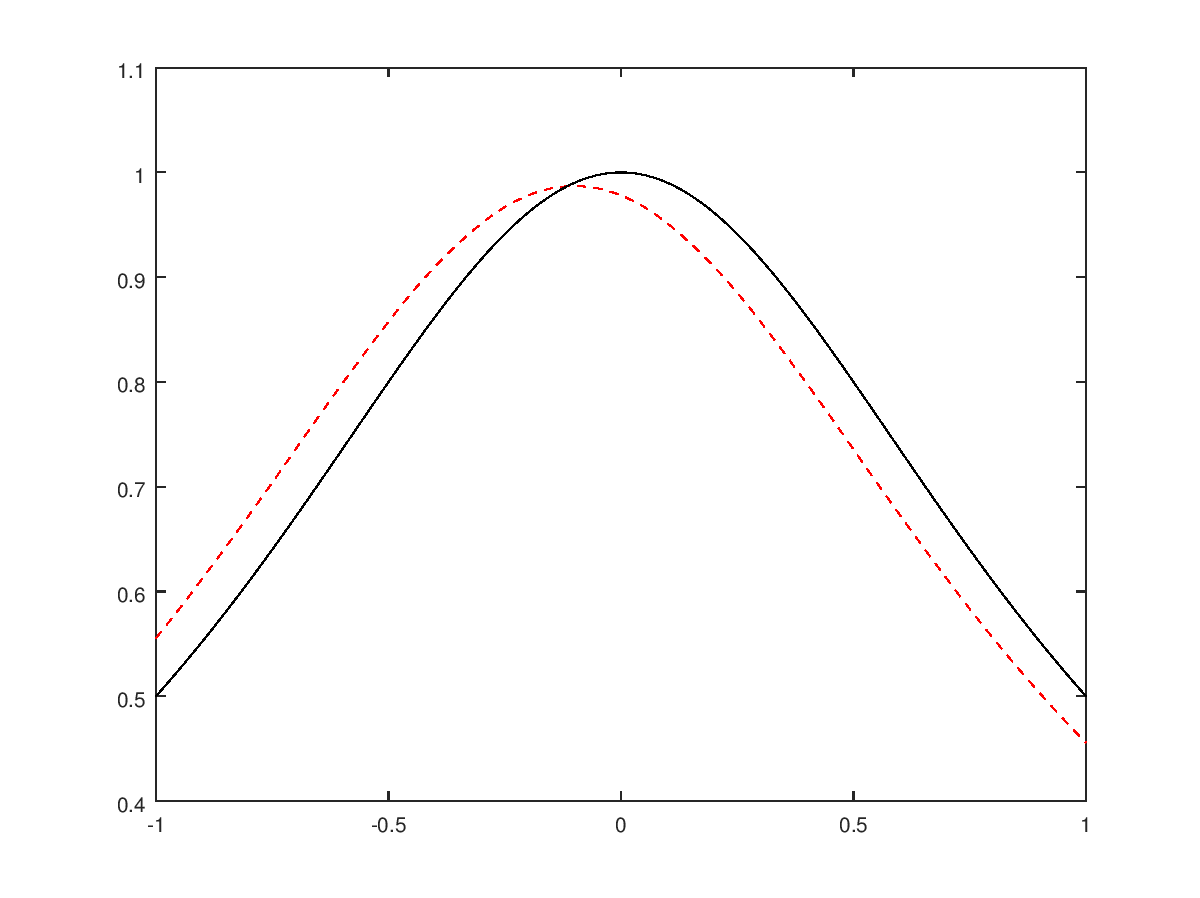}
 \,
\includegraphics[width=6.6cm]{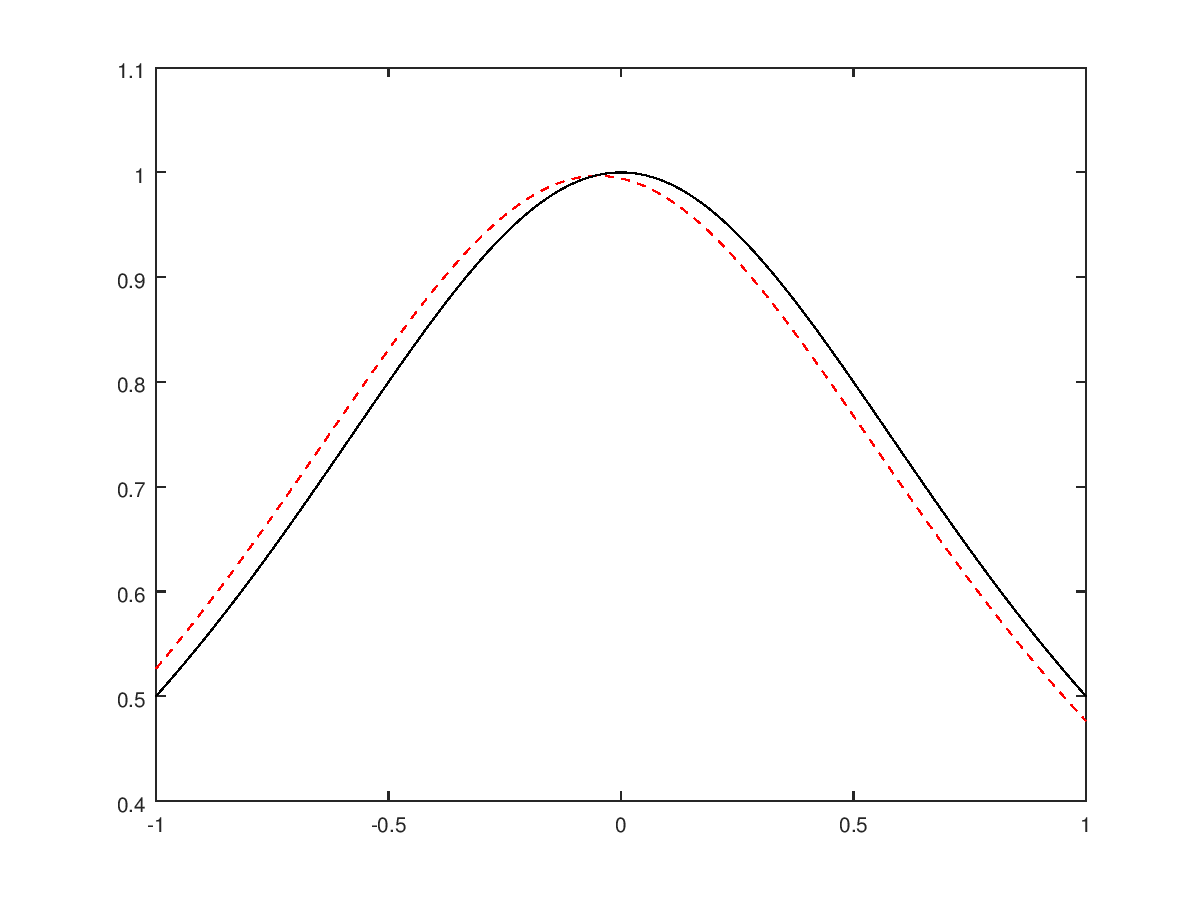}
\caption{The Durrmeyer sampling series $S_w^{\sigma_3,\chi_{[0,1]}}f$ with $w=5$ (on left) and $w=10$ (on right).}
\label{n5}
\end{figure*}

In this case, according to the quantitative estimate given in (Theorem \ref{Quantitativeestimate}), the order of uniform convergence of $S_w^{\sigma_3,\chi_{[0,1]}}f$ to $f$ is at least $\mathcal{O}(\frac{1}{w})$, as $w\rightarrow+\infty$.\\
Moreover, according to the Voronovskaja formula for the Durrmeyer sampling operator established in \cite{2}, it is possible to reach a better order of approximation, in case of sufficiently regular functions and under additional assumptions on the moments. This happens for example choosing $\psi(t)=\frac{1}{2}\chi_{[-1,1]}(t)$, $t\in\mathbb{R}$. In this case the Durrmeyer sampling type series becomes
\[\left (S_w^{\sigma_3,\frac{1}{2}\chi_{[-1,1]}}f \right)(x)=\sum_{k\in\mathbb{Z}}\sigma_3(wx-k) w \int_{\mathbb{R}}\frac{1}{2}\chi_{[-1,1]}(wu-k)f(u)du\]
\[=\sum_{k\in\mathbb{Z}}\sigma_3(wx-k)\frac{w}{2}\int_{\frac{k-1}{w}}^{\frac{k+1}{w}}f(u)du,\,\, \,x\in\mathbb{R}.\]
Thus, since $f\in C^2(\mathbb{R})$ and applying the Voronovskaja formula, we obtain that the order of approximation of $S_w^{\sigma_3,\frac{1}{2}\chi_{[-1,1]}}f$ to $f$, is at least $\mathcal{O}(\frac{1}{w^2})$, as $w\rightarrow+\infty$.
Below are the graphs of the Durrmeyer sampling series $S_w^{\sigma_3,\frac{1}{2}\chi_{[-1,1]}}$ of $f$ for $w=5$ and $w=10$ (see Figure \ref{n10}).
\begin{figure*}[tbph]
\centering
 \includegraphics[width=6.6cm]{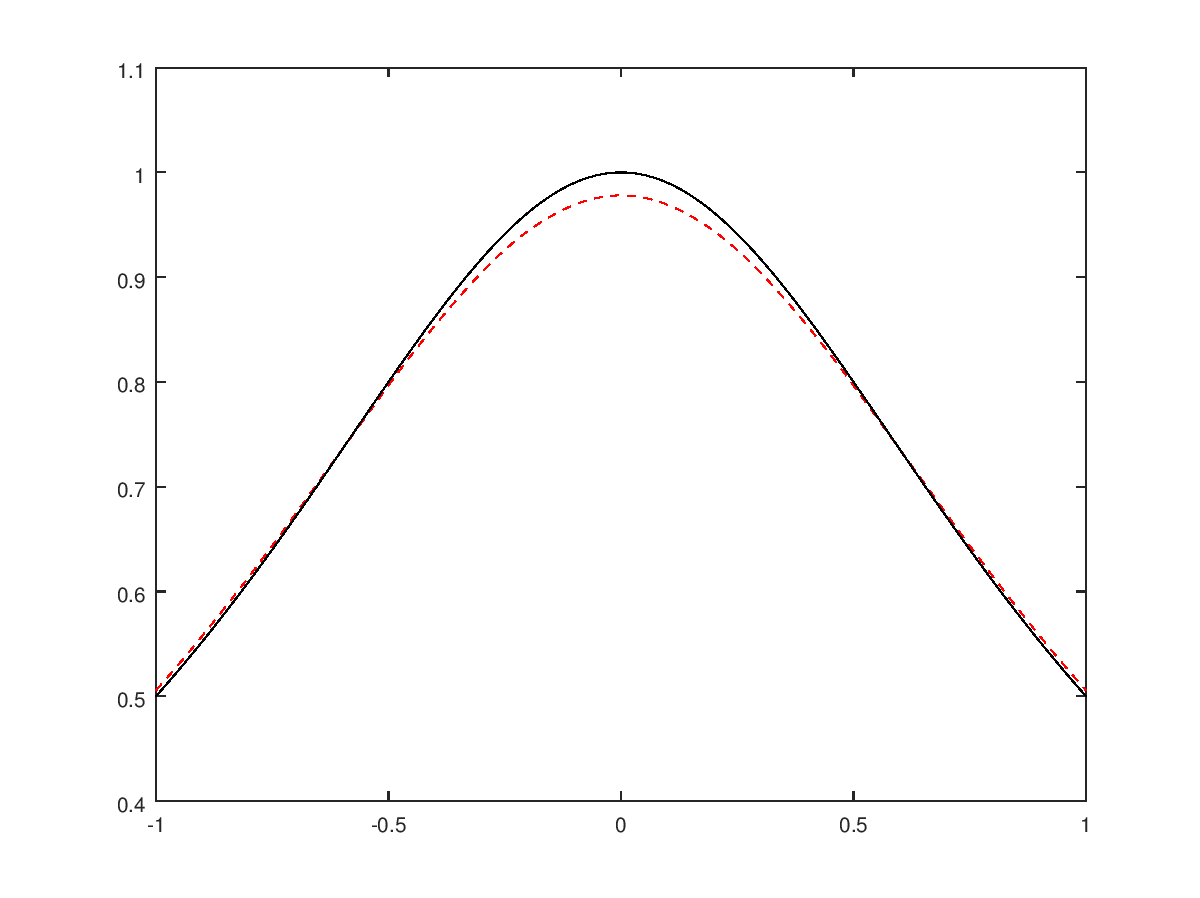}
 \,
\includegraphics[width=6.6cm]{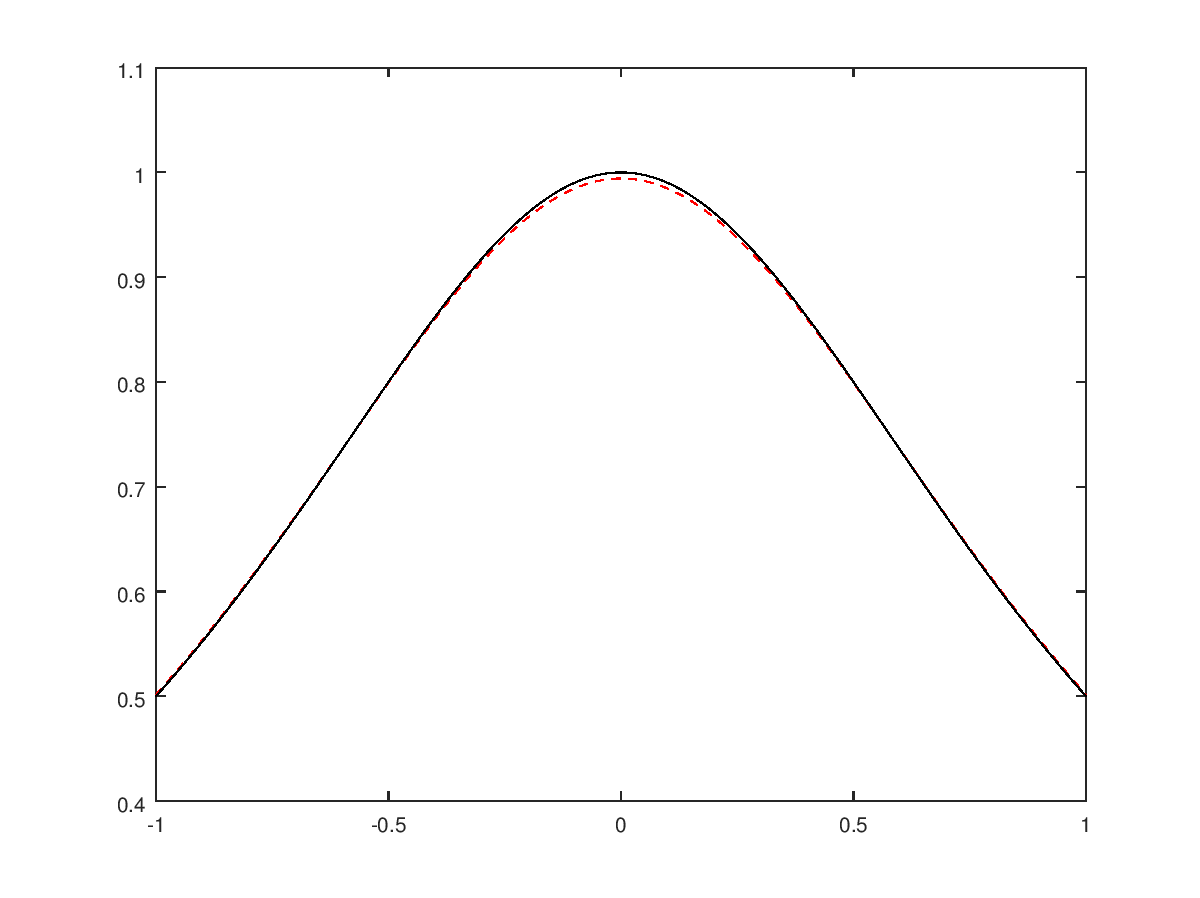}
\caption{The Durrmeyer sampling series $S_w^{\sigma_3,\frac{1}{2}\chi_{[-1,1]}}f$ with $w=5$ (on left) and $w=10$ (on right).}
\label{n10}
\end{figure*}

\item Now, in order to show an application of Theorem \ref{modularconvergence}, a useful example can be given considering functions that are not necessarily continuous and belonging to some $L^p$-space. Hence, choosing
\[\varphi(t)=\sigma_2(t)=(1-\left|t\right|)\chi_{[-1,1]}(t),\qquad t\in\mathbb{R},\] where $\sigma_2$ is the central B-spline of order 2 (see Figure \ref{graficosigma2}) and $\psi(t)=\chi_{[0,1]}(t)$, $t\in\mathbb{R}$, we want to apply the Durrmeyer sampling series $S_w^{\sigma_2,\chi_{[0,1]}}$ to two different discontinuous functions (see Figure \ref{fgradino}), namely,
\[f_1(x):=
\begin{cases}
1, & \left| x \right|\le 1, \\
0, & \left| x \right|> 1, \\
\end{cases}
\]
as well $f_2$, defined by
\[f_2(x):=
\begin{cases}
\frac{9}{x^2}, & x < -1, \\
2, & -1\le x <0, \\
1, & 0\le x <1,\\
\frac{-50}{x^4}, & x\ge1.
\end{cases}
\]
 \begin{figure}[h!]
    \begin{center}
    \includegraphics[width=6.6cm]{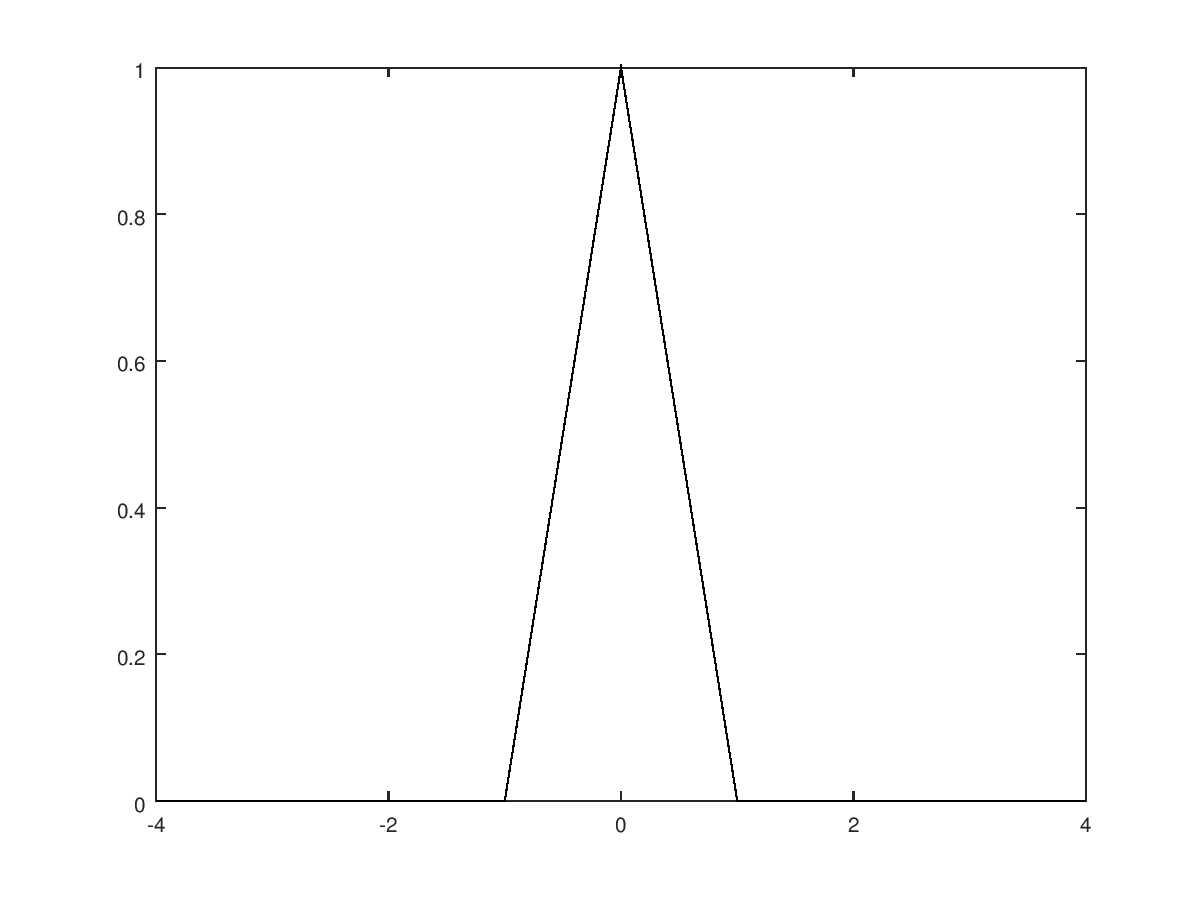}
    \end{center}
    \caption{The spline function $\sigma_2$.}
    \label{graficosigma2}
\end{figure}
\begin{figure*}[tbph]
\centering
 \includegraphics[width=6.6cm]{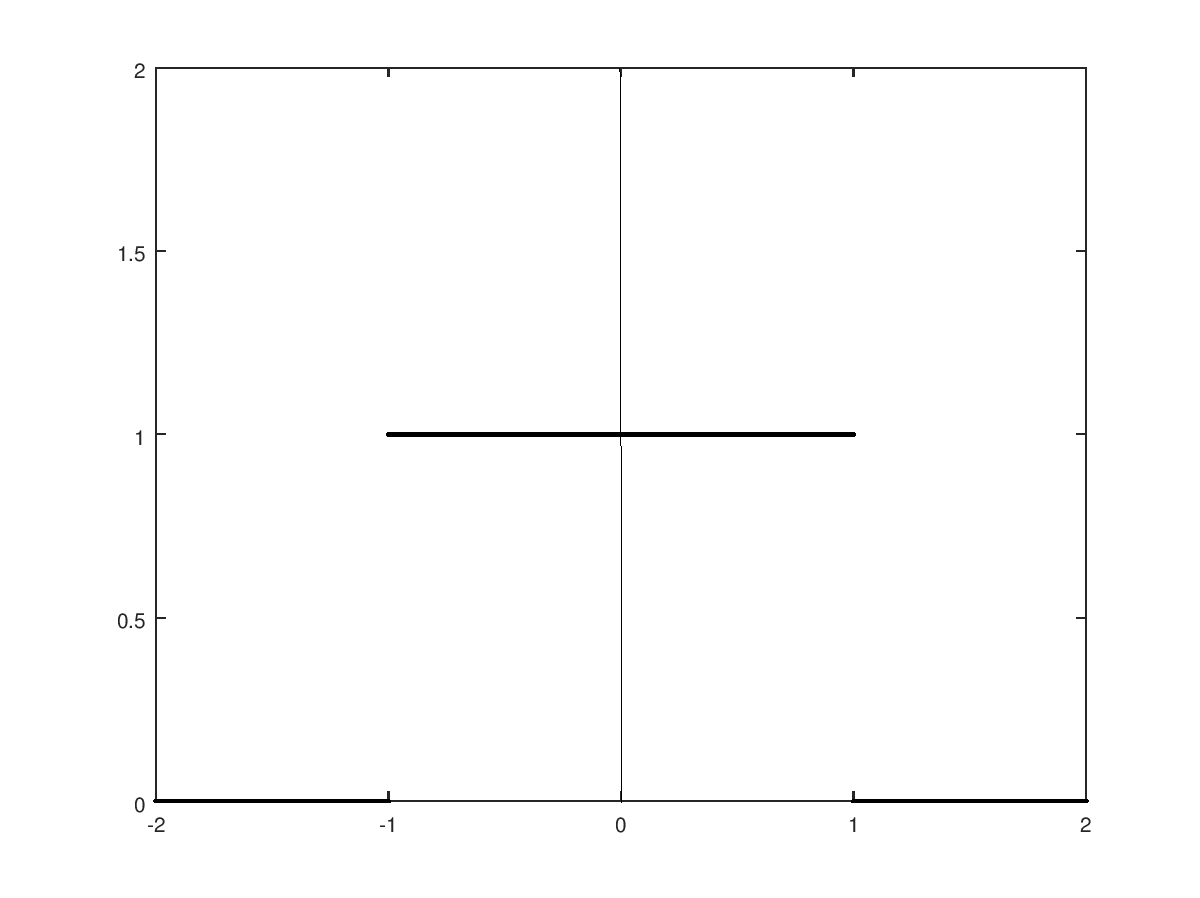}
 \,
\includegraphics[width=6.6cm]{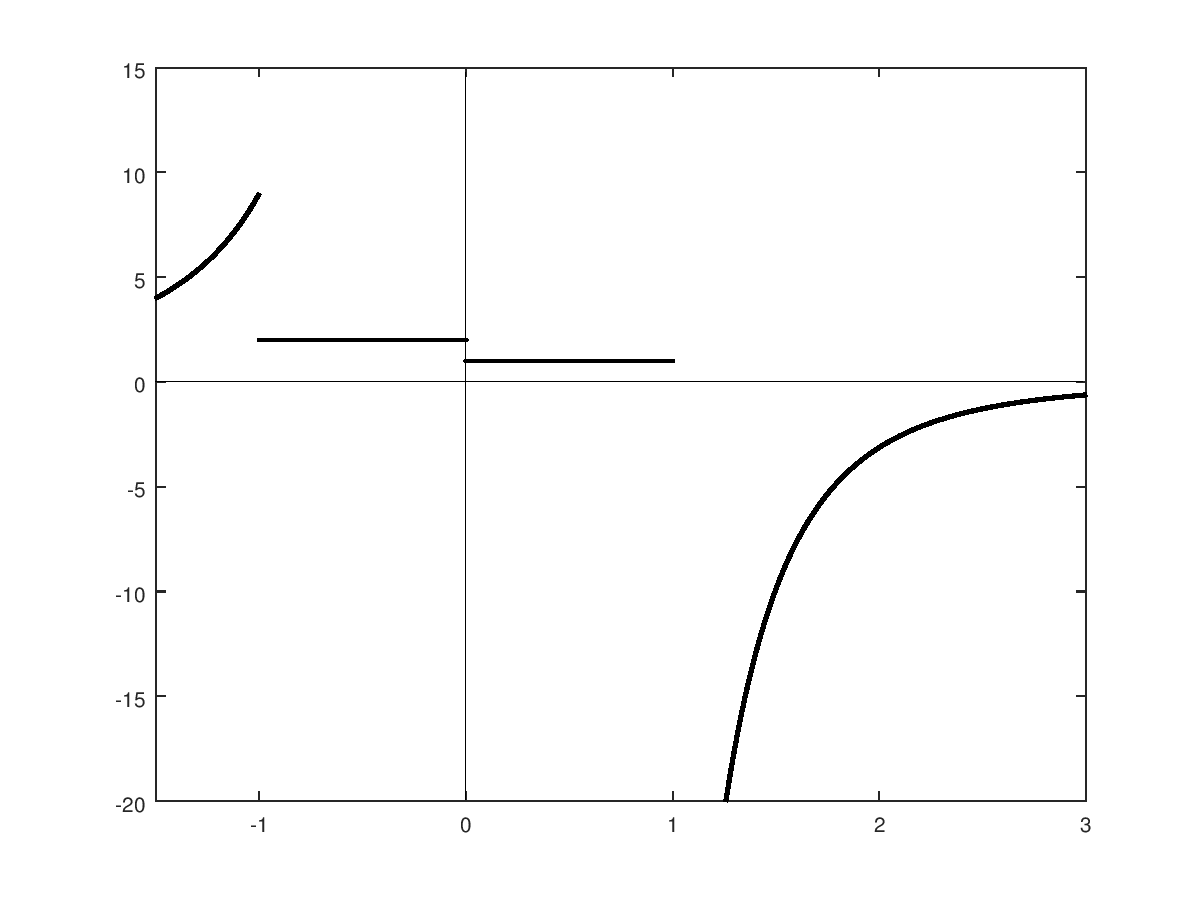}
    \caption{The graphs of the functions $f_1$ and $f_2$.}
    \label{fgradino}
\end{figure*}

  In general, we want to underline that from the properties of the kernel $\varphi=\sigma_n$ and since $\psi=\chi_{[0,1]}$ (which has compact support) satisfies trivially the condition (c) of Remark \ref{osservazione3.1}, Corollary \ref{lp}, Corollary \ref{log} and Corollary \ref{exp} hold.\\
The Durrmeyer type sampling series with $w=5$ and $w=10$ of the functions $f_1$ and $f_2$ are given in Figure \ref{f1} and Figure \ref{f2}, respectively. As before, the red dotted lines represent the graphs of the operators $S_w^{\sigma_2,\chi_{[0,1]}}f_1$ and $S_w^{\sigma_2,\chi_{[0,1]}}f_2$, while the black lines denote the graphs of the functions $f_1$ and $f_2$. Finally, it should also be noted that, since both the kernels $\varphi$ and $\psi$ have compact support, for the evaluation of $S_w^{\sigma_2,\chi_{[0,1]}}f(x)$ for a specific $x\in\mathbb{R}$, only a finite number of mean values are needed, both in the case of functions with compact support, like $f_1$, as in the case of functions with unbounded support, like $f_2$.
\begin{figure*}[tbph]
\centering
 \includegraphics[width=6.6cm]{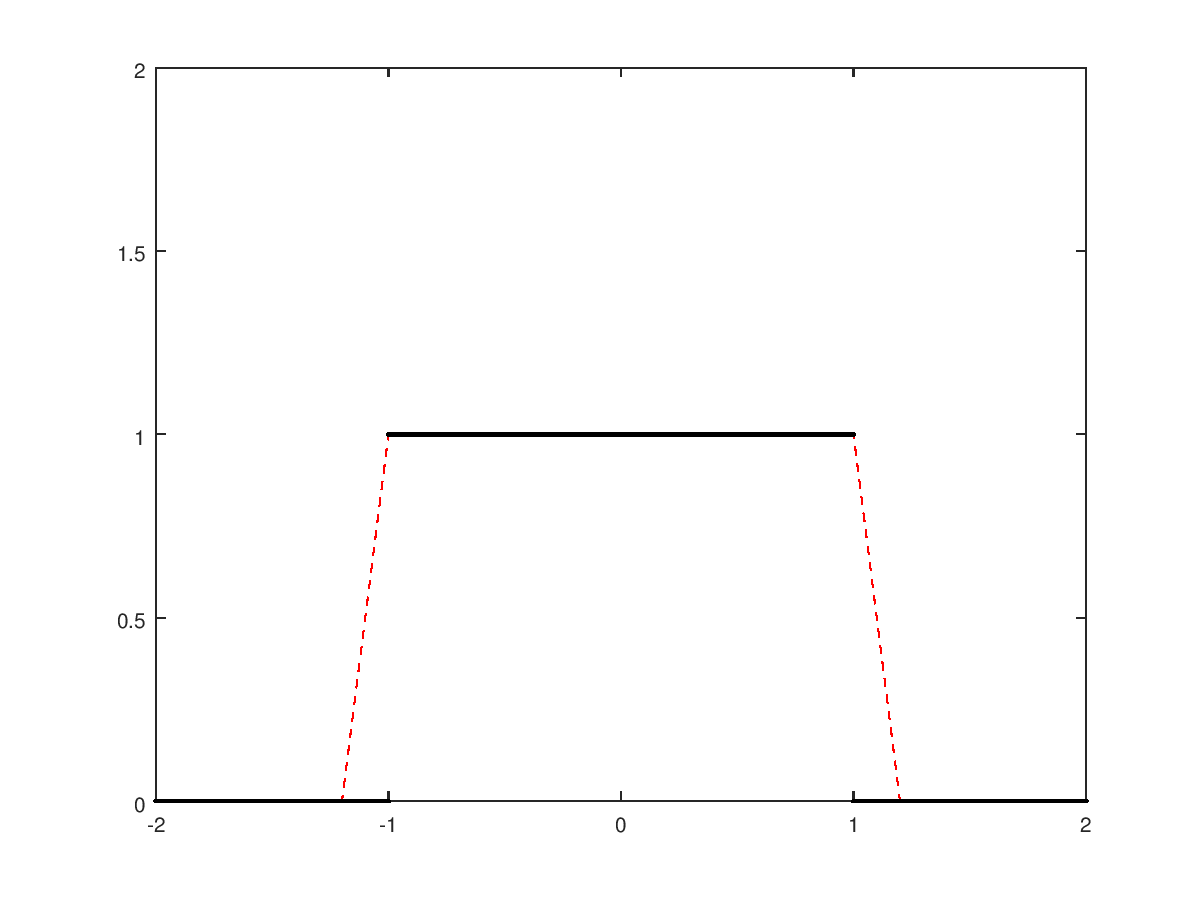}
\includegraphics[width=6.6cm]{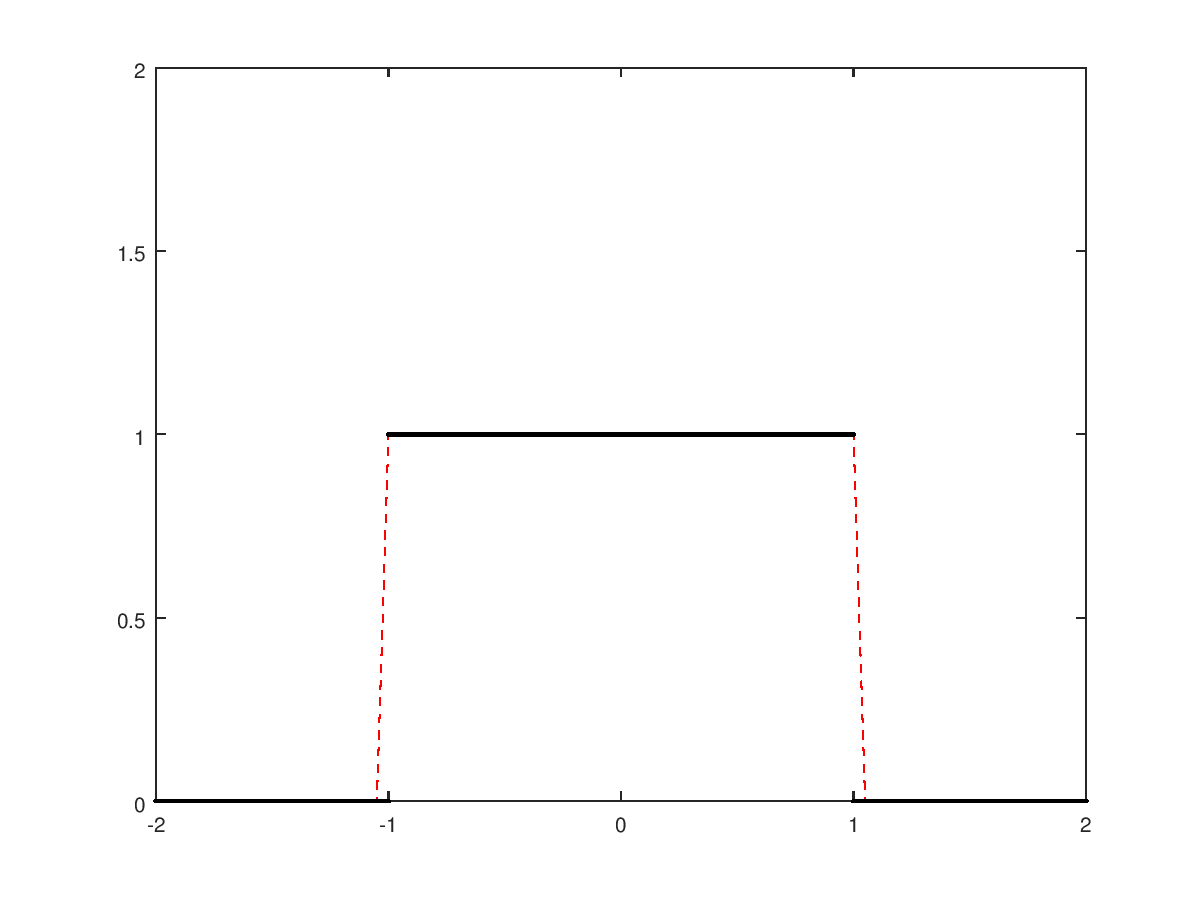}
\caption{The Durrmeyer sampling series $S_w^{\sigma_2,\chi_{[0,1]}}f_1$ with $w=5$ (on left) and $w=10$ (on right).}
\label{f1}
\end{figure*}
\begin{figure*}[tbph]
\centering
 \includegraphics[width=6.6cm]{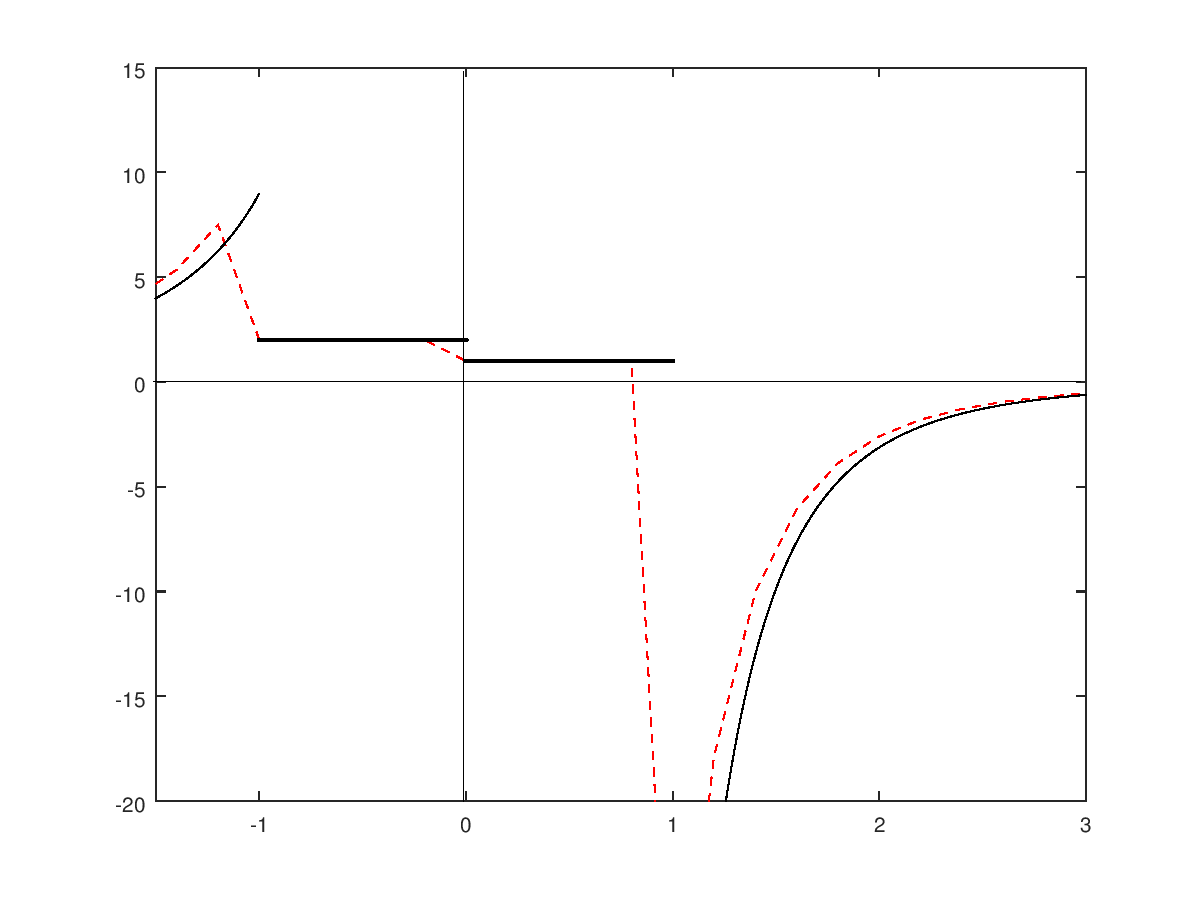}
\includegraphics[width=6.6cm]{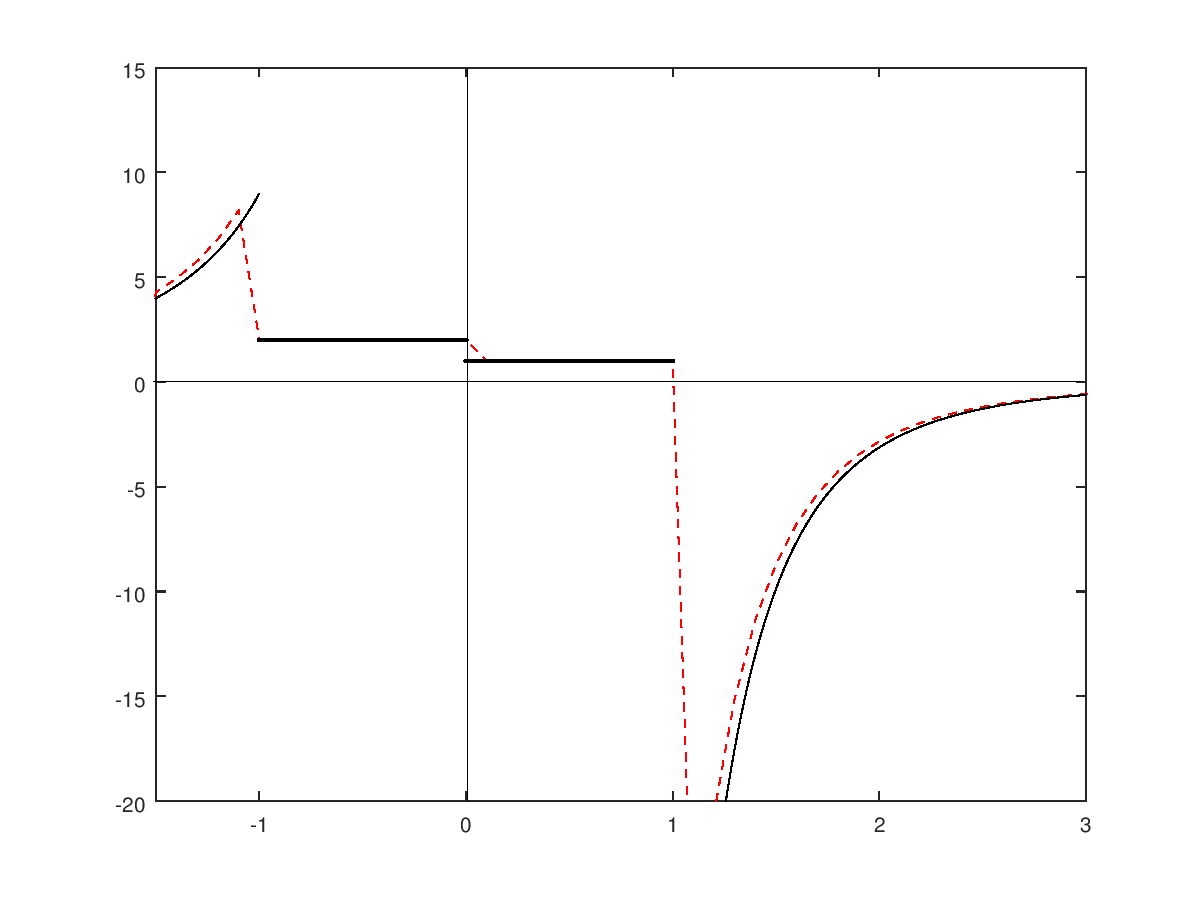}
\caption{The Durrmeyer sampling series $S_w^{\sigma_2,\chi_{[0,1]}}f_2$ with $w=5$ (on left) and $w=10$ (on right).}
\label{f2}
\end{figure*}
\end{enumerate}

  In conclusion, in order to underline that the convergence results proved in this paper hold for a large class of kernels $\varphi$ and $\psi$, we observe that it is possible to provide examples of Durrmeyer sampling operators based on a more general kernel $\psi$, also with unbounded support. For example, we can choose as $\psi$ the Fejér kernel (see Figure \ref{fejer}), defined by
\[F(t):=\frac{1}{2}\text{sinc}^2\left(\frac{t}{2}\right),\,\,t\in\mathbb{R}.\]
 
\begin{figure}[h!]
    \begin{center}
    \includegraphics[width=6.6cm]{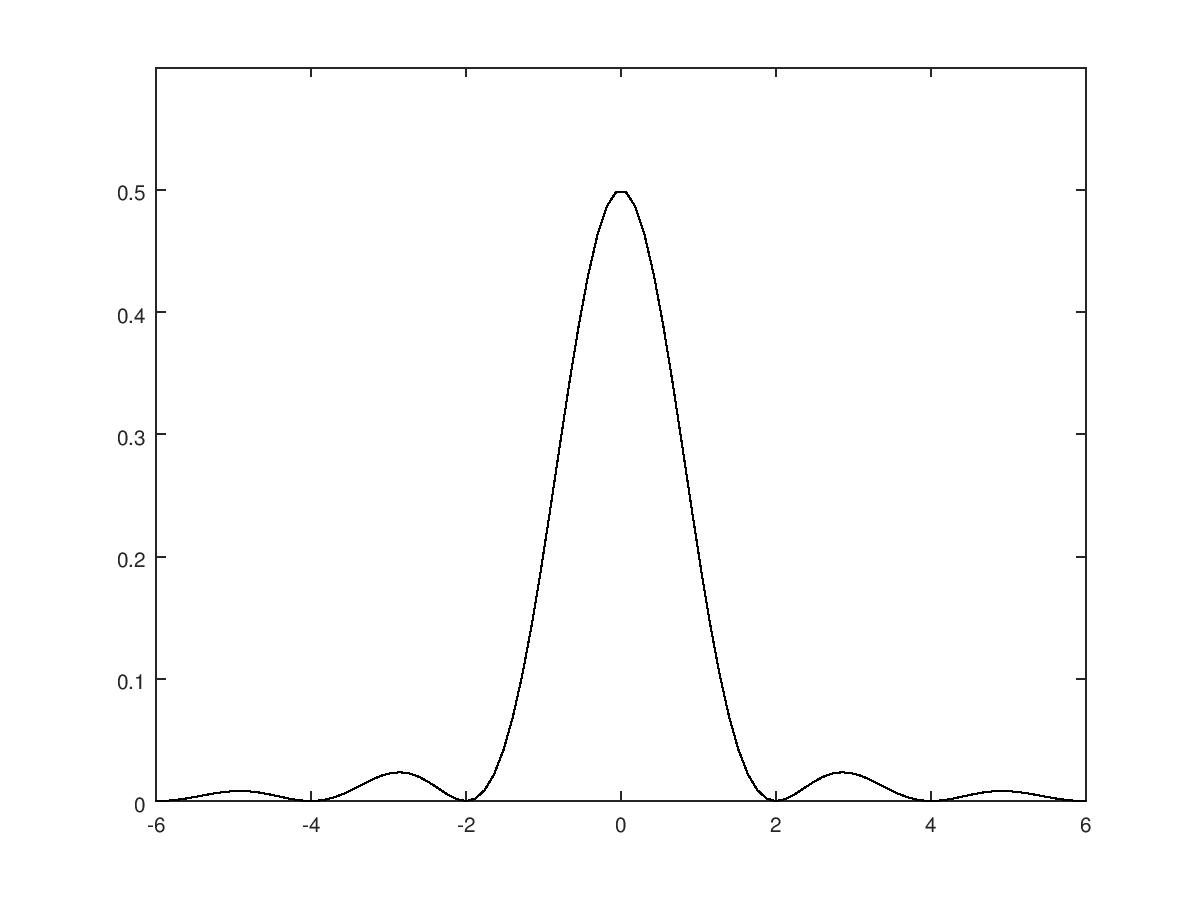}
    \end{center}
    \caption{The Fejér kernel $F$.}
    \label{fejer}
\end{figure}
Obviously, $F$ is bounded and non-negative on $\mathbb{R}$, belongs to $L^1(\mathbb{R})$ and satisfies $\int_{\mathbb{R}}F(t)dt=1$. Moreover, the moment condition $M_0(\psi)<+\infty$ is trivially fulfilled in view of Remark \ref{osservazione3.1} (c) with $0<\nu<1$.\\
Finally, it is interesting to observe that the Fejér kernel can be chosen also as the kernel $\varphi$. Indeed, since its Fourier transform is given by
\[\widehat{F}(v):=
\begin{cases}
1-\left|\frac{v}{\pi}\right|, & v\le\pi, \\

0, & v>\pi,
\end{cases}
\]
(see, e.g., \cite{1971ap}), it follows, by the equivalent condition (\ref{poisson}) (applied to $F$ in place of $\sigma_n$), that $F$ satisfies the discrete singularity assumption (\ref{IpotesiD}) on $\varphi$.

\section*{Acknowledgments}
{\small The authors are members of the Gruppo Nazionale per l'Analisi Matematica, la Probabilit\`a e le loro Applicazioni (GNAMPA) of the Istituto Nazionale di Alta Matematica (INdAM), and of the network RITA (Research ITalian network on Approximation). The first author has been partially supported within the 2019 GNAMPA-INdAM Project ``Metodi di analisi reale per l'approssimazione attraverso operatori discreti e applicazioni'', while the third author within the projects: (1) Ricerca di Base 2018 dell'Universit\`a degli Studi di Perugia - "Metodi di Teoria dell'Approssimazione, Analisi Reale, Analisi Nonlineare e loro Applicazioni", (2) Ricerca di Base 2019 dell'Universit\`a degli Studi di Perugia - "Integrazione, Approssimazione, Analisi Nonlineare e loro Applicazioni", (3) "Metodi e processi innovativi per lo sviluppo di una banca di immagini mediche per fini diagnostici" funded by the Fondazione Cassa di Risparmio di Perugia (FCRP), 2018, (4) "Metodiche di Imaging non invasivo mediante angiografia OCT sequenziale per lo studio delle Retinopatie degenerative dell'Anziano (M.I.R.A.)", funded by FCRP, 2019.}

\end{document}